\documentclass[a4paper,10pt]{scrartcl}
\usepackage[utf8]{inputenc}
\usepackage{amsmath,amssymb,amsthm}
\usepackage[left=2cm, right=2cm]{geometry}
\usepackage{dsfont}
\usepackage{comment}
\usepackage{esint}
\usepackage{color}
\usepackage{tikz}
\usepackage{enumitem}
\usepackage{multicol}
\usepackage{mathtools}
\usepackage[nocompress]{cite}
\usepackage{epsfig}
\usepackage{epstopdf}

\numberwithin{equation}{section}

\mathtoolsset{showonlyrefs}
\usetikzlibrary{decorations.pathreplacing}

\newcommand{\R}{\mathbb{R}}
\newcommand{\N}{\mathbb{N}}

\newcommand{\Z}{\mathbb{Z}}

\newcommand{\Hm}{\mathcal{H}}

\newcommand{\eps}{\varepsilon}
\newcommand{\del}{\partial}
\renewcommand{\div}{\operatorname{div }}
\newcommand{\1}{\mathds{1}}
\newcommand{\loc}{\mathrm{loc}}

\DeclareMathOperator{\supp}{supp}
\DeclareMathOperator{\dist}{dist}

\newcommand{\weakstar}{\stackrel{\ast}{\rightharpoonup}}

\theoremstyle{definition}
\newtheorem{definition}{Definition}[section]
\newtheorem{remark}{Remark}[section]
\theoremstyle{theorem}
\newtheorem{theorem}{Theorem}[section]
\newtheorem*{theorem*}{Theorem}
\newtheorem{proposition}{Proposition}[section]
\newtheorem{lemma}{Lemma}[section]

\author{Janusz Ginster, Peter Gladbach}
\title{Solvation in the Large Box Limit}
\date{}

\begin{document}

\maketitle

\abstract{In this paper, the authors study the limit of a sharp interface model for the solvation of charged molecules in an implicit solvent as the number of solute molecules and the size of the surrounding box tend to infinity. 
The energy is given by a combination of local terms accounting for the physical presence of the molecules in the solvent and a nonlocal electrical energy with or without an ionic effect.
In the presence of an ionic effect, the authors prove a screening effect in the limit, i.e., the limit is completely localized and hence electrical long-range interactions of the molecules can be neglected.
In the absence of the ionic effect, the authors show that the behavior of the energy depends on the scaling of the number of molecules with respect to the size of the surrounding box.
All scaling regimes are identified and corresponding limit results proved.
In regimes with many solute molecules this limit includes electrical interactions of $H^{-1}$-type between the molecules.}

%\keywords{$\Gamma$-convergence \and Molecular solvation \and Poisson-Boltzmann theory \and Phase-field}

\section{Introduction}
In \cite{dai2016convergence}, Dai, Li, and Lu derive a sharp interface model for the solvation of charged molecules in an implicit solvent (see also \cite{PhysRevLett.96.087802,dzubiella2006coupling,ChDzLiMcZh14,ChChDzLiMcZh12,LiZh12,LiLi15} and references therein).
The free energy for $N$ charged molecules at fixed positions $x_1,\dots,x_n$ in a container $\Omega \subseteq \R^3$ is given by
\begin{align}\label{eq: energyF}
 F(x_1,\dots,x_n,u) = \beta \int_{\Omega} (1-u(x)) \,dx + \gamma |Du|(\Omega) +  \int_{\Omega} u(x) \sum_{i=1}^n U^i_{LJ}(x-x_i) \,dx + F_{el}(x_1,\dots,x_N,u),
\end{align}
where the phase-field $u:\Omega \to \{0,1\}$ determines the region occupied by the solvent, $u^{-1}(1)\subseteq \Omega$.
The first term in the energy reflects the amount of work needed to create a solute region in a solvent medium at hydrostatic pressure $\beta$, 
the second term accounts for the interfacial energy between solute and solvent regions where $\gamma$ is the effective surface tension,
and the third term reflects the interaction between the charged molecules and the solvent given by an interaction via a Lennard-Jones $U_{LJ}^i$ depending on the molecule species.\\  
The electrical energy $F_{el}(x_1,\dots,x_N,u)$ is the free energy induced by the charged molecules and the solvent:
\[
 F_{el}(x_1,\dots,x_n,u) \coloneqq \int_{\Omega} \left(-\frac{\eps(u)}2 |\nabla \psi|^2 + Q_{x_1,\ldots,x_n} \psi - u B(\psi)\right) \, dx, \label{eq: sharpinterface}
\]
where $Q_{x_1,\ldots,x_n}\in L^1(\Omega)$ is the total charge density of all solute molecules, and $\psi$ is the electric potential solving the Poisson-Boltzmann equation (see \cite{li09,HoSh90,FoBr97}),
\[
 \begin{cases}
  -\div (\eps(u) \nabla \psi) + u B'(\psi) = \rho &\text{in } \Omega, 
  \\ \psi = \psi_{\infty} &\text{on } \del \Omega, 
 \end{cases}
\]
for some given fixed  $\psi_{\infty}$.
The dielectric constant $\eps(u)$ is given by $\eps_1\approx 80$ for water and $\eps_0\approx 1$ for vacuum. \\
The term $B(\psi)$ models the ionic effect penalizing high electric potentials, and is given by
\begin{equation} \label{eq: defB}
 B(s) \coloneqq k_B T \sum_{k=1}^M c_j^{\infty} (e^{-\frac{q_k s}{k_B T}} - 1),
\end{equation}
where $k_B$ is the Boltzmann constant, $T$ is the temperature, $c_k^{\infty}$ is the bulk concentration, and $q_k$ is the charge of the $k$th ionic species in the solvent, with $\sum_{k=1}^M q_k = 0$. \\
As $B$ is convex, we can write equivalently
\begin{align}\label{eq: electricasmax}
 F_{el}(x_1,\dots,x_n,u) = \max_{\psi \in H_0^1(\Omega)} \int_{\Omega} \left(-\frac{\eps(u)}2 |\nabla \psi|^2 + Q_{x_1,\ldots,x_n} \psi - u B(\psi)\right) \, dx.
\end{align}
Using convex duality one can show that the electrical energy as written in \eqref{eq: electricasmax} equals the free electrical energy associated to the free ions in the solvent and the charges induced by the solute molecules (see also \cite{FoBr97}). 
\newline
\begin{figure}
 \begin{center}
  \includegraphics[scale=3]{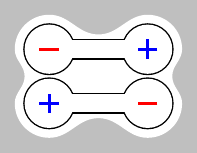}
 \end{center}
\caption{Illustration of the cell problem. 
Molecules will arrange themselves so as to decrease their energy. 
Here the two dipoles are displayed to form a quadrupole, lowering the electrical energy, and share a single bubble, lowering the surface energy.}
\end{figure}
To the best knowledge of the authors, this model has only been studied for a fixed number of solute molecules (see \cite{LiLi15, LiZh12}).
Our contribution is to derive an effective energy in the situation of a large number of solute molecules.
For this, we study  the limiting behavior of a rescaled version of the energy $F$ in the sense of $\Gamma$-convergence (for an introduction see, for example, \cite{braides} or \cite{dalmaso}) as the number of solute molecules and the size of the surrounding box $\Omega$ go to infinity.\\
The main mathematical challenge is the following.
The energy $F$ consists of local terms and the electrostatic the energy, $F_{el}$, which is a priori non-local.
However, the electrostatic energy contains mainly a local self-energy per charge and a nonlocal electrostatic interaction of the charges (see also Subsection \ref{sec: heuristics}). 
In the derivation of the limit energy it is therefore utterly important to distinguish nonlocal effects from the part of the energy which localizes in the limit.
As a main tool, we present a strategy to find clusters of solute molecules whose nonlocal interaction is controlled (see proof of lower bound of Theorem \ref{thm: GammaB} and Lemma \ref{lemma: cubes}). \\ 
We consider two different versions of the energy, the case in which $B$ is as in \eqref{eq: defB}, and the case $B=0$. 
For $B$ as in \eqref{eq: defB}, we show that the energy fully localizes in the limit, i.e., the limit energy is given as the self-energy of the diffused limit molecule distribution (see Theorem \ref{thm: GammaB}). 
Here, we can control the nonlocal interactions as the occurrence of $B$ in the Poisson-Boltzmann equation leads to a fast decay of the electric field $\psi$ generated by a given charge distribution.
This indicates that the ionic effect gives rise to a screening effect, i.e., the local arrangement of the solvent blocks the electric long range interactions (see also \cite{capet2009minimum}). \\
In the case $B=0$ the decay of the electric field $\psi$ is much slower which leads to a competition of the local and nonlocal terms depending on the number of solute molecules.
It shares some structural properties with two-dimensional linearized models for dislocations (see \cite{CeLe05,dLGaPo12, Po07,GaLePo10}), in which an energy of the form
\[
\int_{\Omega} \mathbb{C} \beta: \beta \, dx \text{ subject to } \operatorname{curl }\beta = \sum_i \xi_i \delta_{x_i}
\]
is studied, where $\mathbb{C} \in \R^{2\times2\times2\times2}$ is a linearized elastic tensor.
Ignoring the local terms in $u$, we observe that after partial integration the energy $F$ is essentially the integral of the squared gradient of a potential whose divergence is prescribed, thus playing in our case the counterpart to the $\operatorname{curl}$ in the dislocation model. 
In both cases this leads to a decay of the gradient of the electric potential and the elastic strain away from the solute molecules and dislocations, respectively, which is as fast as the gradient of the fundamental solution to the Poisson equation.
Clearly, due to the difference in dimension, the occurring scales of the problems are different.
However, the limiting behavior is similar as a local and a nonlocal term compete, with the local term being predominant in dilute regimes, regimes with relatively few solute molecules, (see Theorem \ref{thm: GammanoBsub} and  \cite{dLGaPo12, Po07}), and in regimes with more solutes the nonlocal $H^{-1}$-interaction of the solutes dominating (see Theorem \ref{thm: Bsuper} and  \cite{GaLePo10}). 
In dilute regimes, the key to control the electrostatic interactions is a quantitative estimate of the average interaction between different clusters (see Lemma \ref{lemma: cubes}).
In the intermediate, so-called critical, regime both effects are of the same order and appear in the limit (see Theorem \ref{thm: Bcrit} and \cite{GaLePo10}).  \\

\subsection{Setting of the Problem}
Let $N\in\mathbb{N}$ be the number of solute species.
For $r>0$ we define the admissible solute distributions as a subset of the $\R^N$-valued Radon measures, $\mathcal{M}(\Omega;\R^N)$,
\begin{align}
 \mathcal{A}_r(\Omega) \coloneqq \left\{\rho = \left(\sum_{j=1}^{m_i} \delta_{x_j^i}\right)_{i=1}^N \in \mathcal{M}(\Omega;\R^N)\,:\,m_1,\ldots,m_N \in \N, x_j^i \in \Omega, \sup_{x \in \Omega} |\rho(B_r(x))| \leq M \right\}, \label{def: Ar}
\end{align}
where $M>0$ is a fixed constant.
Moreover, for $\rho \in \mathcal{A}_r(\Omega)$ we write
\begin{align}
 Q_r \rho \coloneqq \sum_{i=1}^N \sum_{j=1}^{m_i} \phi_i\left(\frac{x-x_j^i}{r}\right)/r^3\in \mathcal{M}(\R^3,\R)
\end{align}
for the charge density associated to the molecule distribution $\rho$.
Here, the distributions $\phi_1,\ldots,\phi_N\in L^\infty(\R^3,\R)$ are assumed to have compact support, they represent the charge distributions associated to each solute species, and generalize the simple uniform distributions in \cite{dai2016convergence}. 
The upper bound $M>0$ prevents accumulation of too much charge at scale $r$. 
For later purposes we also define for $\rho \in \mathcal{M}(\Omega;\R^N)$ the measure $Q_0 \rho \coloneqq \sum_{i=1}^N \left(\int_{\R^3} \phi_i \,dx \right) \rho^i$, where
$\rho^i$ denotes the $i$th entry of the vector-valued measure $\rho$.\\
\newline
Now, define the rescaled energy $E_r: \mathcal{M}(\Omega;\R^N) \rightarrow \R \cup \{+\infty\}$ as 
\begin{align}\label{energy definition}
 E_r(\rho) \coloneqq \begin{cases} \underset{u \in L^{\infty}(\Omega;\{0,1\})}{\inf} \left[a |\rho| + r^{-3} \int_{\Omega}  \beta (1-u) +  U_{r,\rho}(x) u(x) \,dx + r^{-2} \gamma |D u|(\Omega) +  E_r^{el}(\rho,u)\right] &\text{ if } \rho \in \mathcal{A}_r(\Omega),
 \\ +\infty &\text{ otherwise},\end{cases}
\end{align}
where $U_{r,\rho} (x) \coloneqq \sum_{i=1}^N \sum_{y \in \supp(\rho^i)} U^i_{LJ}(\frac{x - y}{r})$ for functions $U^i_{LJ}: \R^3 \rightarrow \R$ to be specified below in (A2). \\
In view of \eqref{eq: electricasmax}, we the rescaled electrical energy for a function $u \in L^{\infty}(\Omega;\{0,1\})$ and $\rho \in \mathcal{A}_r(\Omega)$ is given by
\begin{align}\label{electrical energy}
 E^{el}_r(\rho,u) \coloneqq \max_{\psi \in H_0^1(\Omega)} \int_{\Omega} \left(-\frac{\eps(u)}{2r} |\nabla \psi|^2 + (Q_r \rho) \, \psi - u r^{-3}B(\psi)\right) \, dx.
\end{align}
If $\Omega$ is unbounded, we denote by $H_0^1(\Omega)$ the closure of $C^{\infty}_c(\Omega)$ with respect to the $H^1(\Omega)$-seminorm.
Note that these functions are not necessarily in $L^2(\Omega)$. 
However, they are always in $L^{2^*}(\Omega)$. \\ \newline
Note here that, up to the total variation of the measure $\rho$, the energy $E_r(\rho)$ equals the energy as defined in \eqref{eq: energyF}, where $\Omega$ is replaced by $\frac1r \Omega$, and the charge distribution $Q_{x_1,\dots,x_n}$ is given by $r^3 (Q_r\rho)(r \cdot)$.
The term $a |\rho|$, for $a>0$ large enough, ensures coercivity of the energy $E_r$ and non-triviality of the later discussed limit energy, i.e., the limit being $-\infty$ everywhere (see also Subsection \ref{sec: heuristics}.
\\[3ex]
We will assume the following throughout the rest of the paper:
\begin{itemize}
 \item [(L)] $\Omega \subseteq \R^3$ is an open set with Lipschitz boundary.
 \item [(A1)] $\eps(1)\geq \eps(0) >0$, $\beta > 0$, $\gamma \geq 0$, and $a\in\R$ is large enough to make the later introduced self-energy coercive.
 \item [(A2)] For every $i=1,\ldots,N$, $U_{LJ}^i:\R^3\to \R$ is negative outside a ball $B(0,R)$ and is integrable on $\R^3 \setminus B(0,r)$ for all $r>0$.
\end{itemize}
We also assume that $B:\R\to \R$ fulfills either
\begin{itemize}
 \item [(B0)] $B=0$
\end{itemize}
or
\begin{enumerate}[label=(B\arabic{*}), ref=(B\arabic{*})]
 \item \label{B1}  $B\geq 0$, $B(0)=0$, $B$ is strictly convex, and $B(s)\geq cs^2 - \frac{1}{c}$ for some $c>0$,
 \item \label{B2} For every $s\in \R$ and every $p\in \partial^- B(s)$ it holds $ps \geq (1+c)B(s)$ for some $c>0$. 
\end{enumerate}

We note that $B$ defined in \eqref{eq: defB} satisfies the two convexity conditions \ref{B1} and \ref{B2}. 
All functions $B(s)=|s|^p$, $p>1$, satisfy \ref{B2} with $c=p-1$. 
However, strict convexity is not enough to guarantee \ref{B2}, illustrated by the function $s\mapsto |s|\log(1+|s|)$. See Lemma \ref{lemma: dual} for why this condition is useful.

\subsection{Heuristics and Scaling}\label{sec: heuristics}

We now look at an example configuration. 
We assume that the number of solute species is $N=1$, $ \Omega = \R^3$, and there are $M=K^3$ evenly spaced molecules positioned on the lattice points $Z_K\coloneqq [0,1)^3 \cap \frac{1}{K} \Z^3$, each with a positive charge, so that $\rho \coloneqq \sum_{z\in Z_K} \delta_z$ and $Q_r \rho = \sum_{z\in Z_K} r^{-3}\mathds{1}_{B(z,r)}$.

We set $u \coloneqq 1 - \sum_{z\in Z_K} \mathds{1}_{B(z,r)}$ and estimate the energy
\begin{align}
 E_r(\rho,u) \approx aM + \beta M\frac{4\pi}{3} + \gamma M 4\pi + M \int_{\R^3\setminus B(0,1)} U_{LJ}\,dx + E^{el}_r(\rho,u).
\end{align}
We see that all terms except the electrical interaction scale with the number of molecules, since they are largely local. The Lennard-Jones potential has a fast-decaying tail that can be ignored. Note that the Lennard-Jones interaction is negative, and if $|U_{LJ}|$ is too large, the energy will be negative.

We have yet to estimate the electrical energy. We first treat the case where $B=0$, which leads to a linear maximization problem
\[
 E^{el}_r(\rho,u) = \sup_{\psi \in H_0^1(\R^3)}  \int_{\R^3} Q_r\rho \,\psi - \frac{\eps(u)}{2r}|\nabla \psi|^2\,dx.
\]

Since the problem is linear, we may write $\psi(x) = \sum_{z\in Z_K} \frac{1}{r^3}\int_{B(z,r)} G_{r,u}(x,y)\,dy$, where $G_{r,u}(x,y)$ is the Green's function, which behaves as $\frac{r}{\eps(1)|x-y|}$ for $|x-y| \geq r$, so that after an integration by parts
\begin{align*}
 E^{el}_r(\rho,u) & =  \sum_{z\in Z_K} \sum_{z'\in Z_K} \frac{1}{2r^6}\int_{B(z,r)}\int_{B(z',r)} G_{r,u}(x,y)\,dx\,dy\\  
\approx &Me_0 + \frac{2\pi}{3}\sum_{z\neq z'} \frac{r}{\eps(1)|z-z'|}, 
\end{align*}
where $e_0$ is the electrical energy of a single charge in $\R^3$, since for $r$ small enough the boundary effect becomes negligible.

Summing up all interactions leads to
\[
 E^{el}_r(\rho,u) \approx Me_0 + \frac{2\pi}{3\eps(1)} rM^2 \int_{[0,1]^3} \int_{[0,1]^3} \frac{dx\,dy}{|x-y|}.
\]

We see that if $Mr\ll 1$, the self-energy dominates, and if $Mr\gg 1$, long-range Coulombic interactions between like charges dominate.

To make this precise, we introduce a parameter $\alpha(r)\in (0,\infty)$ representing the approximate number of solute molecules. Whenever $|\rho_r|(\R^3) \approx \alpha(r)$, we can expect that either
\[
E_r(\rho_r) \approx \alpha(r)\text{, if }r\alpha(r)\ll 1\text{, or }E_r(\rho_r) \approx r \alpha(r)^2\text{, if }r\alpha(r)\gg 1.
\]

Note that if $|\rho_r|(\R^3)\leq C \alpha(r)$, then the rescaled measures $\frac{\rho_r}{\alpha(r)}\in \mathcal{M}(\R^3,\R^N)$ have a vaguely convergent subsequence, and in that topology we will show the following limit scaling, in the sense of $\Gamma$-convergence (see \cite{braides},\cite{dalmaso}):

For $B=0$, whenever $r\to 0$, $\alpha(r)\to \infty$, we have
\begin{itemize}
 \item The subcritical regime $ \alpha(r) r \to 0$: $\frac{E_r}{\alpha(r)}$ $\Gamma$-converges to a local functional $E_0(\rho) = \int_{\R^3} \varphi\left(\frac{d\rho}{d|\rho|}\right)\,d\rho$ depending on the vector-valued mass density.
 \item The supercritical regime $\alpha(r) r \to \infty$: $\frac{E_r}{\alpha(r)^2 r}$ $\Gamma$-converges to the Coulombic long-range interaction 
 \[
 \frac{1}{2\eps(1)}\|Q_0\rho\|_{H^{-1}}^2 = \frac{1}{2\eps(1)}\int_{\R^3} \int_{\R^3} \frac{Q_0\rho(dx)Q_0\rho(dy)}{|x-y|}
 \]
 among net-charged solute molecules.  
 \item The critical regime $ \alpha(r) r \to c\in \R$: $\frac{E_r}{\alpha(r)}$ $\Gamma$-converges to $E_0 + \frac{1}{2c\eps(1)}\|Q_0\rho\|_{H^{-1}}^2$.

\end{itemize}

If $B$ is instead superquadratic, we may take $B(s)=\frac{1}{2}s^2$ as a representative, so that $\psi$ solves the linear maximization problem
\[
  E^{el}_r(\rho,u) = \sup_{\psi \in H_0^1(\R^3)} \sum_{z\in Z_K} \frac{1}{r^3}\int_{B(z,r)} \psi - \frac{\eps(1)}{2r}|\nabla \psi|^2\,dx - \int_{\R^3 \setminus \bigcup_{z\in Z_K} B(z,r)} \frac{\eps(0)}{2r}|\nabla \psi|^2 + \frac{1}{2r^3}\psi^2 \,dx.
\]

Again, $\psi$ is given by the convolution of $Q_r \rho$ with a Green's function $G_{r,u,B}(x,y)$, with this time decays as $G_{r,u,B}(x,y) \approx \frac{r}{\eps(1)|x-y|} e^{-\frac{|x-y|}{r}}$. Now whenever $M r^3 \ll 1$, i.e. $\min_{z\neq z'\in Z_K} |z-z'|\gg r$, the interaction is exponentially weak.

We expect the $\Gamma$-limit of $\frac{E_r}{\alpha(r)}$, in the topology $\frac{\rho_r}{\alpha(r)}\weakstar \rho$, to be the local self-energy of the mass density $E_0(\rho)$.

\subsection{Main Results}

In Section \ref{sec: B} we prove the $\Gamma$-convergence of the rescaled energies $\{E_r\}_{r>0}$ under certain growth conditions on $B$ which include the ionic effect given in \eqref{eq: defB}.
\begin{theorem}\label{thm: GammaB}
Assume \ref{B1} and \ref{B2}.
Moreover, assume that $\alpha(r) \to \infty$ and $\alpha(r) r^3 \to 0$. Then the functionals $\{E_r/\alpha(r)\}_{r>0}$ $\Gamma$-converge, with respect to vague convergence of the measures $\rho_r/\alpha(r)$, to the limit energy 
$E_0: \mathcal{M}(\Omega;\R^3) \rightarrow [0,\infty]$ defined by

\begin{align}
 E_0(\rho) \coloneqq \begin{cases}
             \int \varphi\left( \frac{d\rho}{d|\rho|} \right)\,d|\rho| &\text{if the vector-valued measure } \rho \text{ is nonnegative in each component}. \\
             +\infty &\text{otherwise}.
            \end{cases}
            \end{align}
\end{theorem}
Here $\varphi$ is a suitably defined subadditive, positively $1$-homogeneous function which can be interpreted as the self-energy of local charge distributions, and will be defined in \eqref{eq: defselfenergy}.\\
As argued in Subsection \ref{sec: heuristics}, we show in Section \ref{sec: noB}, by proving the corresponding $\Gamma$-limit results, that in the case $B=0$ three different scaling regimes arise. \\
In the subcritical regime, Subsection \ref{sec: subcriticalB}, the result is the following.
\begin{theorem}\label{thm: GammanoBsub}
Let $\alpha(r) r \rightarrow 0$ as $r\to 0$ and $B=0$. 
If $a>0$ is large enough, then the rescaled energies $\{E_r/\alpha(r)\}_{r>0}$ $\Gamma$-converges, with respect to vague convergence of the measures $\rho_r/\alpha(r)$, to the energy $E^{sub}:\mathcal{M}(\Omega;\R^N) \rightarrow  [0,\infty]$ defined as
\[
E^{sub}(\rho) \coloneqq \begin{cases}
\int \varphi\left( \frac{d\rho}{d|\rho|} \right)\,d|\rho| &\text{if the vector-valued measure } \rho \text{ is nonnegative in each component}, \\ +\infty &\text{otherwise}.
\end{cases}
\]
\end{theorem}
Again, the subdadditive, positively $1$-homogeneous function $\varphi$ is again defined as in \eqref{eq: defselfenergy}, now for $B=0$. \\
In the critical regime, Section \ref{sec: criticalB}, we need to introduce an additional assumption on the admissible charge distributions. 
We assume that two different charges are separated on a scale $\delta_r$ where $\frac{\delta_r}{r} \rightarrow \infty$ and $\delta_r^3 \alpha(r) \to 0$. \\
Before stating the main result, we need to briefly introduce some notation.
For the $i$-th unit vector in $\R^N$ and a Dirac measure $\delta_0$ in the origin we mean $E_1(e_i \delta_0;\R^3)$ the energy as defined in \eqref{energy definition} for $\Omega = \R^3$ and $\rho = e_i \delta_0$. \\
Now we can state our main theorem for this modified energy $\tilde{E}_r$.
\begin{theorem}\label{thm: Bcrit}
Let $\alpha(r) r \rightarrow \alpha \in (0,\infty)$. 
If $E_1(e_i \delta_0;\R^3) > 0$ then the energies $\{\tilde{E}_r / \alpha(r)\}_{r>0}$ $\Gamma$-converge, with respect to vague convergence of the measures $\rho_r/\alpha(r)$, to the energy $E^{crit}: \mathcal{M}(\Omega;\R^N) \rightarrow [0,\infty]$ defined by
\begin{equation}
E^{crit}(\rho) = \begin{cases} \sum_{i=1}^N E_1(e_i\delta_0;\R^3) |\rho^i|(\Omega) + \frac{\alpha}{2\eps(1)} \| Q_0 \rho \|_{H^{-1}}^2 &\text{ if } Q_0 \rho \in H^{-1}(\Omega;\R^M) \text{ and } \rho^i \text{ is a } \\&\text{ nonnegative measure for all } i=1,\dots,N, \\ +\infty &\text{ otherwise.}
\end{cases}
\end{equation}
Moreover, for sequences $\{\rho_r\}_{r>0}$ with uniformly bounded energies $\{\tilde{E}_r(\rho_r) / \alpha(r)\}_{r>0}$, it holds that $\{\rho_r /\alpha(r)\}_{r>0}$ is vaguely precompact in $\mathcal{M}(\Omega;\R^N)$ and $\{Q_r\rho_r / \alpha(r)\}_{r>0}$ is weakly precompact in $H^{-1}(\Omega)$. 
\end{theorem} 
Finally, in Subsection \ref{sec: supercriticalB} we prove the corresponding result in the supercritical regime, where long-range interaction between charges dominates the energy, as  in \cite{capet2009minimum}.
\begin{theorem}\label{thm: Bsuper}
Let $\alpha(r) r \to \infty$ and $r^3 \alpha(r) \to 0$.
Then it holds:
\begin{itemize}
\item For a sequence $\{\rho_r\}_{r>0} \subseteq \mathcal{M}(\Omega;\R^N)$ with uniformly bounded energies $\left\{\frac{E_r(\rho_r)}{r \alpha(r)^2}\right\}_{r>0}$ there exists $\mu \in H^{-1}$ such that ---up to a subsequence--- $Q_r\rho_r/\alpha(r) \rightharpoonup \mu$ in $H^{-1}(\Omega)$.
\item For a sequence $\{\rho_r\}_{r>0} \subseteq \mathcal{M}(\Omega;\R^3)$ such that $\{Q_r \rho_r/\alpha(r)\}_{r>0} \rightharpoonup \mu \in H^{-1}(\Omega)$, it holds
\[
\liminf_{r\to0} \frac1{\alpha(r)^2r} E_r(\rho_r) \geq \frac{1}{2\eps(1)}\|\mu\|_{H^{-1}}^2.
\]
\item Given $\mu \in H^{-1}(\Omega)$, there exists a sequence $\{\rho_r\}_{r>0}$ such that $Q_r \rho_r \rightharpoonup \mu$ in $H^{-1}(\Omega)$ and 
\[
\limsup_{r\to0} \frac1{\alpha(r)^2r} E_r(\rho_r) \leq \frac{1}{2\eps(1)}\| \mu \|_{H^{-1}}^2.
\]
\end{itemize}
\end{theorem}

\begin{remark}
We remark that the topology we use in most of the results is the vague convergence of measures, because a bound on the energy does not guarantee tightness. 
In fact, solute may accumulate at the boundary or escape to infinity. 
We denote vague convergence of a sequence of measures $\{\rho_n\}_{n\in\mathbb{N}}\subseteq \mathcal{M}(\Omega,\R^N)$ by $\rho_n \weakstar \rho$.
See also Section \ref{sec: conclusion}. \\
We may also allow the solutes' charge distributions to rotate independently of each other. This may decrease the limit self-energy, e.g. for two dipoles, at the cost of more cumbersome notation, but will not cause any mathematical difficulties, since $SO(3)$ is compact.
\end{remark}

We start by proving some preliminary results, explain the condition \ref{B2}, and introduce the self-energy density $\varphi$ in Section \ref{sec: preliminaries} below.

\section{Preliminaries and the Self-Energy} \label{sec: preliminaries}

\subsection{Minimax Arguments}

We now define an unmaximized unminimized energy for $\rho \in \mathcal{A}_r(\Omega)$, $\psi \in H_0^1(\Omega) \cap L^{\infty}(\Omega)$, and $u \in L^{\infty}(\Omega,[0,1])$ such that $\int_{\Omega} U(\frac xr) u(x) \,dx < \infty$, by 
\begin{align}
 E_r(\rho,u,\psi) \coloneqq &a|\rho|(\Omega) + r^{-3} \beta \int_{\Omega} (1-u) \,dx + r^{-2} \gamma | D u |(\Omega) + r^{-3} \int_{\Omega} U_{r,\rho}(x) u(x) \,dx \\ &+ \int_{\Omega} \left(\rho \psi - \frac{\eps(u)}{2r} |\nabla \psi|^2 -u r^{-3} B(\psi)\right) \,dx.
\end{align}
We set
\begin{align}
 E_r(\rho,u) \coloneqq \sup_{\psi\in H_0^1(\Omega)} E_r(\rho,u,\psi)
\end{align}
and, finally,
\begin{align}\label{energy new E}
 E_r(\rho) \coloneqq \inf_{u\in L^\infty(\Omega,[0,1])} E_r(\rho,u).
\end{align}
Also, we localize $E_r$ by considering, for $A \subseteq \R^3$, $E_r(\rho,u,\psi;A)$ which is defined as in \eqref{energy new E} after replacing $\Omega$ by $A$. \\
Note that at this point, we seemingly have two definitions \eqref{energy definition} and \eqref{energy new E} for $E_r(\rho)$. 
However, the following lemma will clear up this ambivalence.

\begin{lemma}\label{minimax lemma}
Assume $\gamma\geq 0$, $\Omega \subseteq \R^3$ open with Lipschitz boundary and $B$ satisfies \ref{B1}.
Then for every pair $(\rho,u) \in \mathcal{A}_r(\Omega) \times L^{\infty}(\Omega,[0,1])$ there exists a unique maximizer $\psi \in H_0^1(\Omega)$ of $E_r(\rho,u,\cdot)$. 
Moreover, for each $\rho \in \mathcal{A}_r(\Omega)$ there is a measurable minimizer $u:\Omega \to \{0,1\}$ of the energy $E_r(\rho,\cdot)$. 
In particular, the definitions \eqref{energy new E} and \eqref{energy definition} coincide.\\ 
Also, 
 \[
\min_u \max_\psi E_r(\rho,u,\psi) = \max_\psi \min_u E_r(\rho,u,\psi).   
 \]
\end{lemma}

\begin{proof}

Note that the energy $E_r(\rho, \cdot, \psi)$ is convex and lower semi-continuous with respect to $L^1$-convergence in $u$.
On the other hand, the energy $E_r(\rho,u,\cdot)$ is concave and upper semi-continuous with respect to weak $H^1$-convergence in $\psi$.
By Poincar\'{e}'s inequality and standard estimates, the energy $E_r(\rho,u,\cdot)$ is coercive in $\psi$ uniformly in $u$, i.e., there exists a weakly compact subset $K$ of $H^1_0(\Omega)$ such that for all $u\in L^{\infty}(\Omega,[0,1])$ the optimal $\psi$ lies in $K$.
Hence, we can write 
\[
\min_{u \in L^{\infty}(\Omega,[0,1])} \max_{\psi \in H_0^1(\Omega)} E_r(\rho,u,\psi) = \min_{u \in L^{\infty}(\Omega,[0,1])} \max_{\psi \in K} E_r(\rho,u,\psi), 
\]
and the right hand side satisfies the requirements of the minimax theorem (see \cite{Sion1958}) which yields
\[
 \inf_{u \in L^{\infty}(\Omega;[0,1])} \max_{\psi \in K} E_r(\rho,u,\psi) = \max_{\psi \in K} \inf_{u \in L^{\infty}(\Omega,[0,1])} E_r(\rho,u,\psi)
\]
and this proves the second claim of the lemma. \\
It remains to prove that there exists a minimizer of $E_r(\rho,\cdot)$ with values in $\{0,1\}$. 
By the minimax theorem above, it suffices to prove that for every fixed $\psi$ there exists a minimizing $u$ with values in $\{0,1\}$.
For fixed $\psi \in H^1(\Omega)$, write 
\[
f_{\psi}(x) \coloneqq -\beta r^{-3} + r^{-3} U_{r,\rho}(\frac{x}r) + \frac{\eps(1) -  \eps(0)}{2r} |\nabla \psi(x)|^2.  
\]
Then $u$ minimizes $E_r(\rho,\cdot, \psi)$ in $L^{\infty}(\Omega;[0,1])$ if and only if $u$ minimizes in the same class of functions the energy
\[
 F_{\psi}(u) = \int_{\Omega} f(x) u(x) \, dx + \gamma r^{-2} |D u|(\Omega).
\]
This energy has a minimizer in $L^{\infty}(\Omega,[0,1])$.
If $\gamma = 0$ it is simply given by $\1_{\{f\geq 0\}}$ whereas in the case $\gamma>0$ we can apply the direct method of the calculus of variations. \\
In the imaging community, it is well-known that also for $\gamma>0$ there exists a minimizer which takes only the extreme values $0$ and $1$ (see \cite{ChEsNi06}).
Indeed, by the coarea-formula we can rewrite the energy of a minimizer $u$ as
\[
 F_{\psi}(u) = \int_0^1 \int_{\Omega} f(x) \1_{\{u(x) > t\}} \, dx + \gamma r^{-2} |D \1_{\{u > t\}}|(\Omega) \, dt = \int_0^1 F_{\psi}(\1_{\{u > t\}}) \, dt.
\]
In particular, there exists a $t \in [0,1]$ such that $F_{\psi}(\1_{\{u >t\}}) \leq F_{\psi}(u)$.
 \end{proof}

We now show that the maximizing electric potential decays fast away from $\rho$ even in the nonlinear case.

\begin{lemma}\label{psidecay}
Assume that $B:\R\to \R$ is convex with minimum at $0$.
Let $\rho \in \mathcal{A}_r(\Omega)$, and let $u:\Omega\to \{0,1\}$ be measurable. 
Let $\psi\in H^1_0(\Omega)$ be the maximizer of
\begin{align}
 \int_{\Omega} \left(Q_r\rho \psi -\frac{\eps(u)}{2r}|\nabla \psi|^2 -\frac{u}{r^3}B(\psi)\right)\,dx,
\end{align}
and let $\overline \psi\in H^1_0(\Omega)$ be the maximizer of the linear problem
\begin{align}
 \int_{\Omega} \left((Q_r\rho)^+ \overline\psi - \frac{\eps(u)}{2r}|\nabla \overline\psi|^2\right)\,dx.
\end{align}
Then $\psi \leq \overline\psi$ almost everywhere in $\Omega$.
\end{lemma}

\begin{proof} Note that, by the maximum principle, $\overline\psi \geq 0$.
 Let $A\coloneqq\{x\in \Omega\,:\,\psi(x)>\overline\psi(x)\}$. 
 Then by the respective maximalities of $\psi$ and $\overline \psi$ we have
 \begin{align}
 \int_A \left((Q_r\rho)^+\overline\psi - \frac{\eps(u)}{2r}|\nabla \overline \psi|^2\right)\,dx  
 \geq &\int_A \left((Q_r\rho)^+\psi - \frac{\eps(u)}{2r}|\nabla \psi|^2\right)\,dx\\
 = &\int_A \left(Q_r\rho \psi - \frac{\eps(u)}{2r}|\nabla \psi|^2 - \frac{u}{r^3}B(\psi) + \frac{u}{r^3}B(\psi) + ((Q_r\rho)^+ - Q_r\rho)\psi\right)\,dx\\
 \geq & \int_A \left(Q_r\rho \overline\psi - \frac{\eps(u)}{2r}|\nabla \overline\psi|^2 - \frac{u}{r^3}B(\overline\psi) + \frac{u}{r^3}B(\overline\psi)+ ((Q_r\rho)^+ - Q_r\rho)\psi\right) \,dx \label{eq: comparison}\\
 =&\int_A \left((Q_r\rho)^+\overline\psi - \frac{\eps(u)}{2r}|\nabla \overline \psi|^2\,dx + \int_A ((Q_r\rho)^+ - Q_r\rho)(\psi-\overline\psi)\right)\,dx,
 \end{align}
 where in \eqref{eq: comparison} we used the fact that  $B(\psi)\geq B(\overline\psi)$ in $A$. 
 The last term in the last line is however nonnegative, so that all terms must actually be equal. In particular
 \begin{align}
  \int_A \left( (Q_r \rho) \psi -\frac{\eps(u)}{2r}|\nabla \psi|^2\right)\,dx = \int_A \left((Q_r \rho) \overline\psi -  \frac{\eps(u)}{2r}|\nabla \overline \psi|^2 \right)\,dx,
 \end{align}
which, by the Lax-Milgram theorem, is only possible if $\psi=\overline\psi$ almost everywhere in $A$.
\end{proof}

\subsection{The Significance of Condition \ref{B2}}

We are now able to show that the convexity condition \ref{B2} on $B$ allows us to bound the dual energy of the maximizer $\psi$. Note that for any convex function $B:\R \to \R$ with $B(0)=0$, for any $s\in \R$ and any $p\in \partial^- B(s)$, we have by the definition of the subgradient that $ ps \geq B(s) - B(0) = B(s)$. Assuming instead $ps \geq (1+c)B(s)$ for some $c>0$ is thus a slightly stronger condition than convexity. Note that $s\in \R$ maximizes $B^\ast(p) = \sup_s ps - B(s)$ if and only if $p\in \partial^-B(s)$. If $B$ fulfills condition \ref{B2}, we may estimate
\[
 B^\ast(p) = ps - B(s) \geq (1+c)B(s) - B(s) = cB(s),
\]
i.e., the  primal energy $B(s)$ is bounded by the dual of its subgradient $B^\ast(p)$. This inequality translates to the electrical energy:

\begin{lemma}\label{lemma: dual} 
Assume \ref{B1} and \ref{B2}. Let $r>0$, $\Omega \subset \R^3$ open, Lipschitz bounded.

Let $\rho \in \mathcal{A}_r(\Omega)$, $u\in L^\infty(\Omega,\{0,1\})$, and let $\psi\in H_0^1(\Omega)$ be the unique maximizer of
\[
 E^{el}_r(\rho,u) = \sup_{\psi\in H_0^1(\Omega)} \int_{\Omega} Q_r\rho \psi - \frac{\eps(u)}{2r}|\nabla \psi|^2  - \frac{u}{r^3}B(\psi)\,dx.
\]

Then
\begin{align}\label{convexgrowthcondition}
  \int_{\Omega} \frac{\eps(u)}{2r}|\nabla \psi|^2 + u B(\psi)\,dx \leq  \frac{1}{\min(c,1)}E_r^{el}(\rho,u).
\end{align}
\end{lemma}

\begin{proof}
 We show that the maximizer $\psi$ solves the differential inclusion $Q_r \rho + \div(\frac{\eps(u)}{r}\nabla \psi)\in \frac{u}{r^3}\partial^- B(\psi)$ almost everywhere in $\Omega$, in the sense that the left-hand side is in fact an $L^\infty$ function.
 
 To see this, first note that by Lemma \ref{psidecay} we get $\psi\in L^\infty$, since $Q_r\rho \in C_c^\infty(\Omega)$. By condition \ref{B1}, there is a compact set $K\subset \R$ such that $\partial^- B(\psi(x))\subset K$ for almost every $x\in \Omega$.
 
 Let $\varphi\in C_c^\infty(\Omega)$ be a test function. Then by dominated convergence
 \begin{align}
  \lim_{\eps \searrow 0} \int_\Omega \frac{uB(\psi + \eps \varphi) - uB(\psi)}{\eps}\,dx = \sup_{p\in L^\infty(\Omega,\partial^- B(\psi(\cdot)))}\int_\Omega up\varphi\,dx.
 \end{align}
Since $\psi$ is the maximizer,
\begin{align}\label{maximizer inequality}
 0 \geq & \lim_{\eps \searrow 0} \frac{E^{el}_r(\rho,u,\psi+\eps\varphi) - E^{el}_r(\rho,u,\psi)}{\eps}\\
  = & \left\langle Q_r\rho + \div(\frac{\eps(u)}{r}\nabla\psi),\varphi\right\rangle_{H^{-1},H_0^1} - \sup_{p\in L^\infty(\Omega,\partial^- B(\psi(\cdot)))}\int_\Omega \frac{u}{r^3}p\varphi\,dx.
\end{align}

If $\{p_k\}\subset L^{\infty}(\Omega,\partial^-B(\psi(\cdot)))$ is a sequence, it is in particular bounded, and has a weak-$\ast$ limit $p\in L^{\infty}(\Omega,\partial^-B(\psi(\cdot)))$. Thus, the set of functions $L^{\infty}(\Omega,\frac{u(\cdot)}{r^3}\partial^-B(\psi(\cdot)))$ is a convex and closed subset of $H^-1(\Omega)$, and since $H_0^1$ is a Hilbert space and in particular reflexive, the Hahn-Banach theorem states in light of \eqref{maximizer inequality} that
\begin{align}
 Q_r\rho + \div(\frac{\eps(u)}{r}\nabla\psi) \in L^\infty(\Omega,\frac{u(\cdot)}{r^3}\partial^-B(\psi(\cdot)))\subset H^{-1}(\Omega).
\end{align}

We pick a function $p\in L^\infty(\Omega, \partial^- B(\psi)\cdot)))$ so that $Q_r \rho + \div(\frac{\eps(u)}{r}\nabla \psi) = \frac{u}{r^3}p$. 
Note that $p$ is unique almost everywhere that $u\neq 0$ and arbirtrary elsewhere. Then by \ref{B2}
\begin{align}
 \int_\Omega Q_r \rho \psi \,dx = & \int_\Omega \frac{u}{r^3}p\psi - \div(\frac{\eps(u)}{r}\nabla \psi)\psi \,dx\\
 \geq & (1+1) \int_\Omega \frac{\eps(u)}{2r}|\nabla \psi|^2\,dx + (1+c) \int_\Omega \frac{u}{r^3} B(\psi)\,dx\\
 \geq & (1+\min(c,1)) \int_\Omega \frac{\eps(u)}{2r}|\nabla \psi|^2 + \frac{u}{r^3} B(\psi)\,dx.
\end{align}
Subtracting the integral $\int_\Omega \frac{\eps(u)}{2r}|\nabla \psi|^2 + \frac{u}{r^3} B(\psi)\,dx$ from both sides of the inequality yields the result.
\end{proof}

\subsection{The Limit Energy} 

Here we define the self-energy density $\varphi$ appearing in the subcritical $\Gamma$-limits \ref{thm: GammaB}, \ref{thm: GammanoBsub}, and \ref{thm: Bcrit}.

\begin{definition}
 We define the function $\varphi: \R^N \rightarrow \R$ by 
 \begin{equation}
   \varphi(\xi) \coloneqq \inf\left\{ \liminf_{k\to \infty} \frac{E_1(\rho_k;\R^3)}{z_k} \,:\,\rho_k \in \mathcal{A}_1, z_k > 0, \rho_k(\R^3) / z_k \to \xi \right\}. \label{eq: defselfenergy}
 \end{equation}
\end{definition}

\begin{lemma}\label{lemma: subadditivity}
 The function $\varphi$ is positively $1$-homogeneous and subadditive. If $a$ is large enough then $\varphi$ is also coercive.
\end{lemma}
\begin{proof}
The positive $1$-homogeneity follows from the definition. The only negative term in the energy is
\begin{align}
 \int_\Omega u(x) U_{1,\rho}(x)^-\,dx.
\end{align}
Now $U_{1.\rho}$ depends linearly on $\rho$, and its negative part is integrable, so that for $\rho \in \mathcal{A}$ and $a=0$ we have $E_1(\rho,\R^3) \geq -c|\rho|(\R^3)$. 
Choosing $a>c$, then $E_1(\rho,\R^3) \geq (a-c)|\rho|(\R^3)$, and the coercivity is inherited by $\varphi$.
If $\rho_k(\R^3)/z_k\to \xi$, then $\frac{\rho_k(\R^3)}{z_k/\alpha} \to \alpha \xi$ for $\alpha>0$, so $\varphi(\alpha\xi)\leq \liminf_{k\to \infty} \frac{E_1(\rho_k,\R^3)}{z_k/\alpha}$, so taking the infimum over all such sequences $\varphi(\alpha \xi)\leq \alpha \varphi(\xi)$. Also $\varphi(\xi) \leq \frac{1}{\alpha} \varphi(\alpha\xi)$, so that $\varphi(\alpha\xi)=\alpha\varphi(\xi)$.

To prove subadditivity, we first show that for any $\rho_1$, $\rho_2\in \mathcal{A}_1$ and any $\delta>0$, there is $\rho\in \mathcal{A}_1$ with $\rho(\R^3)=\rho_1(\R^3) + \rho_2(\R^3)$ and $E_1(\rho,\R^3)\leq E(\rho_1,\R^3) + E(\rho_2,\R^3)+\delta$. 
To see this, consider $R>0$ and set $\rho \coloneqq \rho_1 + T_{Re_1}^\sharp \rho_2$, where $T_{Re_1}^\sharp \rho_2\in \mathcal{A}_1$ is $\rho_2$ translated by $Re_1$. 
We see that $\dist(\supp \rho_1, \supp T_{Re_1}^\sharp \rho_2) \geq R-R_0$ for some $R_0\in \R$ since both $\rho_1$ and $\rho_2$ have compact support, so that, in particular, $\rho \in \mathcal{A}_1$, and clearly $\rho(\R^3)=\rho_1(\R^3) + \rho_2(\R^3)$. 
Now consider the minimizing $u_1,u_2:\R^3\to \{0,1\}$ for $\rho_1$ and $\rho_2$ respectively, and set $u\coloneqq\min(u_1,u_2(\cdot - Re_1))$. 
Then 
\[
|Du|\leq |Du_1| + |Du_2| \text{ and } \int_{\R^3} (1-u)\,dx \leq \int_{\R^3} (1-u_1)\,dx + \int_{\R^3} (1-u_2)\,dx.  
\]
Note that $u_1$ and $u_2$ are equal to $1$ outside of a large ball $B(0,R_0)$. Then
\begin{align}
\int_{\R^3} u U_{1,\rho}\,dx = & \int_{\R^3} u (U_{1,\rho_1} + U_{1,\rho_2})\,dx\\
\geq &\int_{B(0,R/2)} \left(u_1 U_{1,\rho_1} + u_2 U_{1,\rho_2}\right)\,dx - \int_{\R^3\setminus B(0,R/2)} \left(U_{1,\rho_1}^- + U_{1,\rho_2}^-\right)\,dx\\
\geq &\int_{\R^3} \left(u_1 U_{1,\rho_1} + u_2 U_{1,\rho_2}\right)\,dx -C(|\rho_1|(\R^3)+\rho^3|(\R^3))R^{-3},
\end{align}
due to the decay of $U_{1,\rho}$.

This leaves us to estimate the electrical energy. 
Let $\psi\in H_0^1(\R^3)$ be the maximizer. 
By Lemma \ref{psidecay} and the existence of a Green's function that decays as $\frac{1}{|x|}$ (see \cite{Littman1963}) we have
\begin{align}
 |\psi(x)| \leq \frac{C}{|x|} + \frac{C}{|x-Re_1|}
\end{align}
as long as $\min(|x|,|x-Re_1|) \geq \frac{R}{4}$, where $C$ depends on $\rho_1$ and $\rho_2$.

We now choose a large number $0\ll M\ll\frac{R}{4}$, and pick one of the $\lfloor \frac{R}{4M}\rfloor$ annuli $A_i=B(0,\frac{R}{4}+iM)\setminus B(0,\frac{R}{4}+(i-1)M)$ for $i=1,\ldots,\lfloor \frac{R}{4M}\rfloor$, for which
\begin{align}
 \int_{A_i} |\nabla \psi|^2\,dx \leq \frac{1}{\lfloor \frac{R}{4M}\rfloor}\int_{B(0,\frac{R}{2})}|\nabla \psi|^2\,dx,
\end{align}
and define $\psi_1 \coloneqq \psi \eta_i$, where $\eta_i\in C^1_c(B(0,\frac{R}{4}+iM))$ is a cut-off function which is $1$ in $B(0,\frac{R}{4}+(i-1)M)$, with $|\nabla \eta_i|\leq \frac{2}{M}$, $\eta_i\leq 1$. We see that
\begin{align}
 \int_{A_i} |\nabla \psi_1|^2\,dx \leq &2\int_{A_i} \left(|\nabla \psi|^2 + \frac{4}{M^2} \frac{C}{|x|}\right)\,dx\\
 \leq & \frac{4M}{R}\int_{B(0,\frac{R}{2})}|\nabla \psi|^2\,dx + \frac{C}{M}.
\end{align}
We now do the same with an annulus around $Re_1$ and obtain $\psi_2$ which has similar estimates. If we pick $M$ large enough and $R$ even larger, we get that
\begin{align}
 &\int_{\R^3} \left( (Q_1\rho_1) \psi_1 - \frac{\eps(u)}{2r} |\nabla \psi_1|^2 - \frac{u}{r^3}B(\psi_1)\right)\,dx + \int_{\R^3} \left(Q_1(T_{Re_1}\rho_2) \psi_2 - \frac{\eps(u)}{2r} |\nabla \psi_2|^2 - \frac{u}{r^3}B(\psi_2)\right)\,dx\\
 \geq  &\int_{\R^3} Q_1\rho \psi - \frac{\eps(u)}{2r} |\nabla \psi|^2 - \frac{u}{r^3}B(\psi)\,dx - \delta.
\end{align}

It follows that $E^{el}_1(\rho)\leq E^{el}_1(\rho_1)+E^{el}_2(\rho_2)+\delta$.

To show that the limit energy is subadditive, fix $\xi_1,\xi_2\in\R^N$ and assume that $\rho_1^k/z_1^k \to \xi_1$ and $\rho_2^k/z_2^k \to \xi_2$, with $\liminf_{k\to \infty} \frac{E_1(\rho_1^k,\R^3)}{z_1^k} = \varphi(\xi_1)$ and $\liminf_{k\to \infty} \frac{E_1(\rho_2^k,\R^3)}{z_2^k} = \varphi(\xi_2)$. 
Let $z_k\coloneqq k(z^k_1 + z^k_2)$, $n^k_1 \coloneqq \lfloor\frac{z_k}{z^k_1}\rfloor$, $n^k_2 \coloneqq \lfloor\frac{z^k}{z^k_2}\rfloor$, where $\lfloor z \rfloor$ denotes the largest integer less than $z$. 
Note that all three sequences converge to infinity.

By the subadditivity above, there is $\rho_k\in \mathcal{A}_1$ such that $\rho_k(\R^3) = n^k_1 \rho_1^k(\R^3) + n^k_2 \rho_2^k(\R^3)$ and $E_1(\rho,\R^3) \leq n^k_1 E_1(\rho_1^k,\R^3) + n^k_2 E_1(\rho_2^k,\R^3)+1$. Note that
\begin{align}
 \lim_{k\to \infty} \frac{\rho_k(\R^3)}{z_k} = \lim_{k\to\infty} \frac{\rho_1^k}{z_1^k} + \lim_{k\to\infty} \frac{\rho_2^k}{z_2^k} = \xi_1 + \xi_2,
\end{align}
and
\begin{align}
 \lim_{k\to\infty} \frac{E_1(\rho_k,\R^3)}{z_k} = \lim_{k\to \infty} \frac{E_1(\rho_1^k,\R^3)}{z^k_1} + \lim_{k\to \infty} \frac{E_1(\rho_2^k,\R^3)}{z^k_2} = \varphi(\xi_1) + \varphi(\xi_2).
\end{align}

It follows that $\varphi(\xi_1+\xi_2) \leq \varphi(\xi_1) + \varphi(\xi_2)$.
\end{proof}

\begin{remark}
 Let $\bar{\varphi}: \R^N \rightarrow \R$ be defined by
 $\bar{\varphi}(\xi) \coloneqq \inf_{\rho \in \mathcal{A}_1(\R^3): \rho(\R^3) = \xi} E_1(\rho;\R^3)$.
 Then $\varphi$ is the positively $1$-homogeneous, subadditive envelope of $\bar{\varphi}$.
\end{remark}

In the case where we assume well-separateness of different solutes i.e., we additionally assume in the definition of $\mathcal{A}_r$ that admissible measures are single Diracs which have a distance $2 \delta_r \gg r$,
we need the following result.

\begin{lemma}\label{lemma: optimalprofiles}
Let $i \in \{1,\dots,N\}$. 
For the $i$th standard unit vector $e_i \in \R^N$ and a Dirac mass in $0$, $\delta_0$, let $\rho = \delta_0 e_i$.
 Then
\[
 \lim_{R\to\infty}E_r(\rho;B_R(0)) = \lim_{R\to\infty}\min_u \max_{\psi \in H_0^1(B_R(0))} E_r(\rho,u,\psi_R;B_R(0)) = \min_u \max_{\psi \in H_0^1(\R^3)} E_r(\rho,u,\psi;\R^3) = E_r(\rho;\R^3). 
 \]
 %\item If $Q_r \rho$ is radial and $B$ is convex then the minimax pair $(u,\psi)$ for $\min_u \max_{\psi} E_r(\rho,u,\psi;B_{R}(0))$ is radial. 
 %If $\gamma = 0$, the maximizer $\psi_R$ is in $W^{2,\infty}(\R^3)$, and for $\gamma>0$, this holds everywhere except on $J_u$.
 \end{lemma}

\begin{proof}
It always holds that $E_r(\rho;B_R(0)) \leq E_r(\rho;\R^3) + \int_{\R^3} U_{LJ}^i(x) \,dx$.\\ 
For $E_r(\rho;\R^3)$ the optimal $u$ is $1$ outside a certain ball around zero.  
The corresponding optimizer decays as $\frac{1}{|x|}$.
Just as in the previous proof, we may cut-off this function on an annulus of thickness $L>0$ on which the gradient energy is small to produce a competitor 
which has only slightly more energy but is compactly supported.
Hence, for $L \to \infty$ the corresponding energy converges to $E_r(\rho;\R^3)$. 
\end{proof}

\section{The Subcritical Regime for superquadratic $B$}\label{sec: B}

In this section we prove Theorem \ref{thm: GammaB} assuming that $B$ satisfies \ref{B1} and \ref{B2}.

\begin{theorem*}
Assume that $\alpha(r) \to \infty$ and $\alpha(r) r^3 \to 0$. 
Then the functionals $\{E_r/\alpha(r)\}_{r>0}$ $\Gamma$-converge, in the topology $\rho_r/\alpha(r) \weakstar \rho$, to the limit energy 
$E_0: \mathcal{M}(\Omega;\R^3)\rightarrow \R\cup\{\infty\}$ defined by

\begin{align}
 E_0(\rho) = \begin{cases}
             \int_{\Omega} \varphi\left( \frac{d\rho}{d|\rho|} \right)\,d|\rho| &\text{if the vector-valued measure } \rho \text{ is a nonnegative measure in each component,} \\
             +\infty &\text{otherwise},
            \end{cases}
\end{align}
where $\varphi$ is defined in \eqref{eq: defselfenergy}.
\end{theorem*}

\begin{proof}[Proof of the $\Gamma$-$\limsup$]
We first show that finite sums of vector-valued Dirac masses are energy-dense.
Let $\rho \in \mathcal{M}(\Omega,\R^N)$ be any vector-valued Radon measure. For $\delta>0$ define
\begin{align}
 \rho_\delta \coloneqq \sum_{z\in \delta \Z^3 \cap B_{\delta^{-1}}(0)} \rho(Q_{z,\delta}) \delta_z,
\end{align}
where $Q_{z,\delta} \coloneqq z + [-\delta/2,\delta/2)^3$ is the half-open Voronoi cell of $z$. 
For a test function $\phi\in C_c(\Omega,\R^N)$ with modulus of continuity $\omega:[0,\infty)\to [0,\infty)$ we have, for $\delta$ small enough,
\begin{align}
 \left|\langle \phi, \rho_\delta - \rho \rangle\right| \leq &\sum_{z\in \delta \Z^3\cap B_{\delta^{-1}}(0)} \left| \int_{Q_{z,\delta}} \phi \, \cdot d\rho - \phi(z)\cdot\rho(Q_{z,\delta})\right|\\
 \leq & \sum_{z\in \delta \Z^3}  \int_{Q_{z,\delta}} |\phi - \phi(z)| \, d|\rho|\\
 \leq & \omega(\sqrt{3}\delta) |\rho|(\supp \phi).
\end{align}
Since $\lim_{r\to 0} \omega(r) = 0$, this shows that $\rho_\delta \weakstar \rho$.

Now to show that
\begin{align}
 \limsup_{\delta \to 0} E_0(\rho_\delta) \leq E_0(\rho),
\end{align}
by the convexity and positive $1$-homogeneity of $\varphi$, we can use Jensen's inequality for each Dirac mass to obtain
\begin{align}
E_0(\rho) = \sum_{z\in \delta \Z^3} \int_{Q_{z,\delta}} \varphi\left(\frac{d\rho}{d|\rho|}\right)\,d|\rho|\geq \sum_{z\in \delta \Z^3\cap B_{\delta^{-1}}(0)}  \varphi\left( \rho(Q_{z,\delta})\right) =  E_0(\rho_\delta).
\end{align}
Together with the lower-semicontinuity of $E_0$, this shows the energy-density of finite sums of Dirac masses.

To find a competitor to $E_r$ close to $\sum_{i=1}^M p_i \delta(x_i)$, we fix $\delta>0$ and find an almost optimal competitor to the cell problem for $p_i$, i.e., $M$ admissible measures $\rho_1,\ldots,\rho_M\in \mathcal{A}_1$ and $M$ positive numbers $z_1,\ldots,z_M$ such that
\begin{align}
 \sum_{i=1}^M \left|\frac{\rho_i(\Omega)}{z_i} - p_i\right| \leq \delta \text{ and }  \sum_{i=1}^M \left(\frac{E_1(\rho_i)}{z_i} -\varphi(p_i)\right) \leq  \delta.
\end{align}

Find $R>0$ large enough so that all $\rho_i$ are supported inside of $B_R(0)$, all optimal $u_i$ are identically $1$ outside of $B_R(0)$, and
\begin{align}\label{eq: estLJ}
\sum_{i=1}^M \int_{\Omega \setminus B_R(0)} U_{1,\rho_i}\,dx \geq - \delta. 
\end{align}
Note that the optimal $u_i$ is necessarily $1$ in the complement of any convex set $A\subseteq \Omega$ as long as the Lennard-Jones potential $U_{1,\rho_i}$ is negative outside of $A$.

In  order to obtain a competitor $\rho_r\in \mathcal{A}_r$ to $E_r$ such that $\rho_r/\alpha(r)$ is close in the vague topology to  $\sum_{i=1}^M p_i \delta(x_i)$, we replace each $p_i \delta(x_i)$ with $N_i(r)\in \N$ translated copies of $\rho_i(\cdot/r)$ arranged on a lattice.
Choose
\begin{align}\label{chargenumber}
 N_i(r) \coloneqq \lfloor \frac{\alpha(r)}{z_i} \rfloor
\end{align}
and a lattice spacing $s(r)>0$ such that
\begin{align}\label{separation}
 \frac{s(r)}{r} \to \infty,
\end{align}
but
\begin{align}\label{gridsize}
 \max_{i=1,\ldots,M} N_i(r)^{1/3} s(r) \to 0.
\end{align}
Setting
\begin{align}
s(r) \coloneqq \sqrt{\frac{r}{\alpha(r)^{1/3}}}.
\end{align}
satisfies \eqref{separation} and \eqref{gridsize}.
Define
\begin{align}
 \rho_r \coloneqq \sum_{i=1}^M \sum_{z\in Z_{N_i}} \rho_i\left(\frac{\cdot - s(r)z}{r} -x_i\right),
\end{align}
where $Z_{N} \subseteq \Z^3$ is a set of size $N$ with $\max_{z\in \Z_{N}} |z| \leq N^{1/3}$.

Note that for every $i$ the support of
\begin{align}
 \sum_{z\in Z_{N_i(r)}} \rho_i\left(\frac{\cdot - s(r)z}{r} -x_i\right)
\end{align}
is contained in $B(x_i,s(r) N_i(r)^{1/3})$.
By \eqref{gridsize} these balls are pairwise disjoint for $r$ small enough and their radii converge to $0$ as $r\to0$ . 
Also, by the definition \eqref{chargenumber} of $N_i(r)$, we have
\begin{align}
 \frac{\rho_r(B(x_i,s(r) N_i(r)^{1/3}))}{\alpha(r)} = \frac{N_i(r) \rho_i(\Omega)}{\alpha(r)} \to \frac{\rho_i(\Omega)}{z_i},
\end{align}
hence
\begin{align}
 \frac{\rho_r}{\alpha(r)} \weakstar \sum_{i=1}^M \frac{\rho_i(\Omega)}{z_i} \delta(x_i),
\end{align}
which is close to the target distribution.

By \eqref{separation}, the lattice size is much larger than $r$. This allows us to define a global $u_r$ as
\begin{align}
u_r(x) \coloneqq \begin{cases}
u_i\left(\frac{x-s(r)z-x_i}{r}\right) &\text{in }B(x_i+s(r)z,rR) \text{ for } z \in Z_{N_i{r}},\\
1 &\text{elsewhere}.
\end{cases}
\end{align}
Now we want to estimate the energy
\begin{align}\label{eq: localandnonlocal}
 E_r(\rho_r,u_r) = \sup_\psi \bigg[ &\frac{1}{r^3}\int_{\Omega} u_r U_{r,\rho_r}\,dx + \frac{1}{r^2}\gamma |Du_r| + \frac{1}{r^3}\beta \int_{\Omega} (1-u_r)\,dx + a |\rho_r|  \\ 
 &+\int_{\Omega} \left((Q_r \rho_r) \psi - \frac{\eps(u_r)}{2r}|\nabla \psi|^2 -\frac{1}{r^3} B(\psi)u_r \right)\,dx \bigg].
\end{align}
The local terms are nonpositive outside of the balls $B(x_i+s(r)z,rR)$, where $U_{r,\rho_r}<0$. Inside the balls, we have
\begin{align}\label{local energy}
&\frac{1}{\alpha(r)} \sum_{i=1}^M \sum_{z\in Z_{N_i(r)}} \bigg[ \frac{1}{r^3}\int_{B(x_i+s(r)z,rR)} u_r U_{r,\rho_r}\,dx + \frac{1}{r^2}\gamma |Du_r|(B(x_i+s(r)z,rR)) \\
& \hspace{2.5cm} + \frac{1}{r^3}\beta \int_{B(x_i+s(r)z,rR)} (1-u_r)\,dx + a |\rho_r|(B(x_i+s(r)z,rR)) \bigg] \\
= &\frac{1}{\alpha(r)}\sum_{i=1}^M N_i(r) \int_{B_R(0)} u_i U_{1,\rho_i}\,dx + \gamma |Du_i|(B_R(0)) + \beta \int_{B_R(0)} (1-u_i)\,dx + a |\rho_i|(B_R(0)) + \delta\\
 \leq &\sum_{i=1}^M \frac{N_i(r)}{\alpha(r)}\left( \int_{\Omega} u_i U_{1,\rho_i}\,dx + \gamma |Du_i|(\Omega) + \beta \int_{\Omega} (1-u_i)\,dx + a |\rho_i|(\Omega) + \delta\right),
\end{align}
where the inequality stems from the fact that the effect of all solutes outside of $B(x_i+s(r)z,rR)$ on $U_{r,\rho_r}$ is smaller than $\delta$ (see \eqref{eq: estLJ}), and we used the change of variables formula.
In order to show the $\Gamma$-$\limsup$ inequality, it remains to control the nonlocal term in \eqref{eq: localandnonlocal}, namely
\begin{align}
 \sup_{\psi\in H^1_0(\Omega)} &\int_{\Omega} \left((Q_r \rho_r) \psi - \frac{\eps(u_r)}{2r}|\nabla \psi|^2 -\frac{1}{r^3} B(\psi)u_r\right)\,dx\\
 = &\int_{\Omega} \left( (Q_r \rho_r) \psi_r - \frac{\eps(u_r)}{2r}|\nabla \psi_r|^2 -\frac{1}{r^3} B(\psi_r)u_r\right)\,dx.
\end{align}
Here $\psi_r\in H^1_0(\Omega)$ is the unique maximizer, which, due to the convexity property \ref{B2} of $B$, satisfies
\begin{align}\label{H^1 estimate}
 & \int_{\Omega} \left(\frac{\eps(u_r)}{2r}|\nabla \psi_r|^2 + \frac{1}{r^3} B(\psi_r)u_r \right)\,dx\\
 \leq &c \int_{\Omega} \left((Q_r \rho_r) \psi_r - \frac{\eps(u_r)}{2r}|\nabla \psi_r|^2 - \frac{1}{r^3} B(\psi_r)u_r \right)\,dx.
\end{align}

We decompose $\psi_r$ into multiple competitors $\psi_{r,i,z}$, one for each pair $(i,z)$, with $i=1,\ldots,M$, $z\in Z_{N_i(r)}$. To do this, we cut $\psi_r$ off in one of the annuli $B(x_i+s(r)z,r(R+j+1))\setminus B(x_i+s(r)z,r(R+j))$ for $j= 1\ldots, \lceil \delta^{-1} \rceil$, defining
\begin{align}
 \psi_{r,i,z}(x) \coloneqq \psi_r(x) \eta_j\left(\frac{x-x_i - s(r)z}{r}\right),
\end{align}
where $\eta_j \in C_c^\infty(\R^3)$ is a cut-off function with $\eta_j = 1$ in $B(R+j)$, $\eta_j=0$ outside of $B(R+j+1)$, and $\|\nabla \eta_j\|_{L^\infty} \leq 2$.

Since the annuli $B(x_i+s(r)z,r(R+j+1))\setminus B(x_i+s(r)z,r(R+j))$ are pairwise disjoint for different $i,z,j$ for $r$ small enough, we choose one $j\in \{1,\ldots,\lceil \delta^{-1} \rceil\}$ such that
\begin{align}
 \sum_{i=1}^M \sum_{z\in N_i(r)} &\int_{B(x_i+s(r)z,r(R+j+1))\setminus B(x_i+s(r)z,r(R+j))} \left( \frac{\eps(u_r)}{2r}|\nabla \psi_r|^2 +\frac{1}{r^3} B(\psi_r)u_r \right)\,dx\\
 \leq &\delta \int_{\Omega} \left( \frac{\eps(u_r)}{2r}|\nabla \psi_r|^2 +\frac{1}{r^3} B(\psi_r)u_r \right)\,dx\\
 \leq &c\delta \int_{\Omega}  \left( (Q_r\rho_r) \psi_r - \frac{\eps(u_r)}{2r}|\nabla \psi_r|^2 -\frac{1}{r^3} B(\psi_r)u_r \right)\,dx.
\end{align}
We estimate the Dirichlet energies of the cut-off versions $\psi_{r,i,z}$ using \ref{B1} and the product rule, to obtain 
\begin{align}
  \sum_{i=1}^M \sum_{z\in N_i(r)} &\int_{\Omega} \left( \frac{\eps(u_r)}{2r}|\nabla \psi_{r,i,z}|^2 +\frac{1}{r^3} B(\psi_{r,i,z})u_r \right) \,dx\\
  \leq &(1+4c\delta)  \int_{\Omega} \left( \frac{\eps(u_r)}{2r}|\nabla \psi_r|^2 +\frac{1}{r^3} B(\psi_r)u_r \right)\,dx.
\end{align}
Note that the supports of the functions $\psi_{r,i,z}$ are pairwise disjoint. \\
We define $\psi_{i,z}(x) \coloneqq \psi_{r,i,z}(\frac{x-x_i-s(r)z}{r})$, so that by the change of variable formula, 
\begin{align}
\frac{1}{\alpha(r)} &\int_{\Omega} \left((Q_r \psi_r) \rho_r - \frac{\eps(u_r)}{2r}|\nabla \psi_r|^2 -\frac{1}{r^3} B(\psi_r)u_r \right)\,dx\\
 \leq &\frac{1}{\alpha(r)} \sum_{i=1}^M \sum_{z\in Z_{N_i(r)}}\left[ \int_{\Omega} \left((Q_r \rho_r) \psi_{r,i,z} - \frac{\eps(u_r)}{2r}|\nabla \psi_{r,i,z}|^2 -\frac{1}{r^3} B(\psi_{r,i,z})u_r\right) \, dx + 4c \delta \right]\\
 \leq &\tilde{c}\delta + \frac{1}{\alpha(r)}\sum_{i=1}^M \sum_{z\in Z_{N_i(r)}}\left[ \int_{\R^3} \left((Q_1 \rho_i) \psi_{i,z}- \frac{\eps(u_i)}{2}|\nabla \psi_{i,z}|^2 - B(\psi_{i,z})u_i\right) \, dx  \right]\\
 \leq &\tilde{c}\delta +  \sum_{i=1}^M \frac{N_i(r)}{\alpha(r)}\sup_\psi \int_{\R^3} \left((Q_1 \rho_i)\psi - \frac{\eps(u_i)}{2}|\nabla \psi|^2 - B(\psi)u_i \,\right) dx.
\end{align}
Since $\frac{N_i(r)}{\alpha(r)} \to z_i$, combining the last estimate with \eqref{local energy} we obtain
\begin{align}
 \limsup_{r\to 0}\frac{E_r(\rho_r,u_r)}{\alpha(r)} \leq & \sum_{i=1}^M  \varphi(p_i) + \tilde{c}\delta.
\end{align}

Taking a diagonal sequence in $\delta$ we get weak$\ast$ convergence of $\frac{\rho_r}{\alpha(r)}$ to $\sum_{i=1}^M p_i\delta_{x_i}$ and we deduce the upper bound inequality.
\end{proof}

\begin{proof}[Proof the $\Gamma$-$\liminf$]

Here we cannot choose $\rho_r$, but are free to choose $\psi_r$.
The key in the proof will be to group most of the solutes in clusters and to construct an almost optimal $\psi$ for the cluster separately.
The fact that $B$ grows superquadratically will allow us to control the interaction with the remaining solutes using the energy. 

Let $\alpha(r)\to \infty$, $\alpha(r)r^3 \to 0$, and let $\{\rho_r\}_{r>0} \in \mathcal{A}_r$ be a family of $\R^N$-valued finite Radon measures with $\rho_r/\alpha(r) \weakstar \rho$, with $u_r\in L^\infty(\Omega,\{0,1\})$ the corresponding minimizer of $E_r(\rho_r,\cdot)$.
 
 Take $l=r^{1/3}$ and define the cubes around an  offset lattice $l(\Z^3 + z_0)$ as $Q_{z,z_0,l,p}=l(z+z_0) + [-p/2,p/2)^3$ for $z\in \Z^3$, so that for $p=l$ the cubes form a partition of $\R^3$. 
 Define for $z_0\in [0,1)^3$ the set
 \begin{align}
 A_{z_0,r} \coloneqq \Omega \setminus \bigcup_{z\in \Z^3} Q_{z,z_0,r^{1/3},r^{1/3}-r^{2/3}},
 \end{align}
which consists of a periodic arrangement of plates of thickness $r^{2/3}$ oriented in the three cardinal directions with spacing $r^{1/3}$, intersected with $\Omega$.

Choose $z_0\in [0,1)^3$ such that
\begin{align}
 \rho_r(A_{z_0,r})\leq C r^{1/3}\alpha(r).
\end{align}
For every $z\in \Z^3$ we define the Voronoi cube $Q_{z,r} \coloneqq Q_{z,z_0,r^{1/3},r^{1/3}}$ and the slightly smaller cube $q_{z,r} \coloneqq Q_{z,z_0,r^{1/3},r^{1/3}-r^{2/3}}$. 
Set $\overline{u_r}:\R^3\to\{0,1\}$ as
\begin{align}
 \overline{u_r} \coloneqq \begin{cases}
            u_r &\text{in }Q_{z,z_0,r^{1/3}-r^{2/3} + Lr},\\
            1 &\text{elsewhere},
           \end{cases}
\end{align}
%Peter Anfang
where $L>0$ is chosen such that $U_{r,\overline{\rho}_{r}} \leq 0$ outside $Q_{z,z_0,r^{1/3}-r^{2/3} + Lr}$ for all $z \in \Z^3$.

We modify $\rho_r$ by removing charges close to $\partial\Omega$ and in $A_{z_0,r}$, with $\rho_{z,r} \coloneqq \rho_r \llcorner Q_{z,z_0,r^{1/3}-r^{2/3}}$ and  $\overline{\rho_r} \coloneqq \sum_{z \in Z^3\,:\,Q_{z,r}\subset \Omega} \rho_{z,r}$.

We want to replace $E_r(\rho_r,u_r)$ with $E_r(\overline{\rho_r},\overline{u_r})$. Note that $\overline{u_r} \geq u_r$, which by itself decreases most terms in the energy, except for the Lennard-Jones term $\frac{1}{r^3}\int_{\Omega} uU_{r,\rho_r}\,dx$ and the total variation $\frac{\gamma}{r^2}|Du|(\Omega)$.

The total variation of $\overline{u_r}$ can be estimated by
\begin{align}
 |D\overline{u_r}|(\Omega)\leq & |Du_r|( \bigcup_{z \in \Z^3} Q_{z,z_0,r^{1/3}-r^{2/3} + Lr}) + \|Tu_r - 1\|_{L^1(\bigcup_{z \in \Z^3} \partial Q_{z,z_0,r^{1/3}-r^{2/3} + Lr})} \\
 \leq & |Du_r|(\Omega) + \Hm^2(B_{z_0,r}),
\end{align}
where $Tu_r\in L^1_\loc(\bigcup_{z \in \Z^3} \partial Q_{z,z_0,r^{1/3}-r^{2/3} + Lr})$ is the trace of $u_r$ from inside the cubes $Q_{z,z_0,r^{1/3}-r^{2/3} + Lr}$, and $B_{z_0,r}$ is the set of all points on $\bigcup_{z \in \Z^3} \partial Q_{z,z_0,r^{1/3}-r^{2/3} + Lr}$ with $u=0$ on the line segment pointing outwards up to $\partial Q_{z,z_0,r^{1/3}}$,
\begin{align}
 B_{z_0,r} \coloneqq \left\{x\in \bigcup_{z \in \Z^3} \partial Q_{z,z_0,r^{1/3}-r^{2/3} + Lr}\,:\,u_r(x+t\nu_x)=0\textrm{ for }\Hm^1\textrm{-almost all } t \in[0,\frac12 r^{2/3}]\right\},
\end{align}
$\nu_x$ being the outer unit normal to $\partial Q_{z,z_0,r^{1/3}-r^{2/3} + Lr}$, which is well-defined $\Hm^2$-almost everywhere. 
Using Fubini's theorem, we estimate the $\Hm^2$-measure of $B_{z_0,r}$ by
\begin{align}
 \Hm^2(B_{z_0,r}) \leq & 2 r^{-2/3} \int_{A_{z_0,r}} (1-u_r)\,dx,
\end{align}
so that
\begin{align}\label{TV estimate modified}
 \frac{\gamma}{r^2}|D\overline{u_r}|(\Omega)\leq &\frac{\gamma}{r^2}|Du_r|(\Omega) + 2\gamma\frac{r^{1/3}}{r^3} \int_{\Omega}(1-u_r)\,dx. 
\end{align}

The Lennard-Jones term can be treated using the linearity of $U_{r,\rho}$ with respect to $\rho$,
\begin{align}\label{LJ estimate modified}
 \frac{1}{r^3}\int_{\Omega} \overline{u_r} U_{r,\overline{\rho_r}} - u_r U_{r,\rho_r}\,dx = & \frac{1}{r^3}\left(\int_{\Omega} (\overline{u_r} - u_r) U_{r,\overline{\rho_r}} - \int_{\Omega} u_r U_{r,\rho_r - \overline{\rho_r}}\,dx\right)\\
 \leq &|\rho_r - \overline{\rho_r}|(\Omega)\int_{\R^3} U_{LJ}^-(x)\,dx\\
 \leq &C r^{\frac13} \alpha(r),
\end{align}

which is much smaller than $\alpha(r)$.

\begin{figure}
 \begin{center}
 \includegraphics{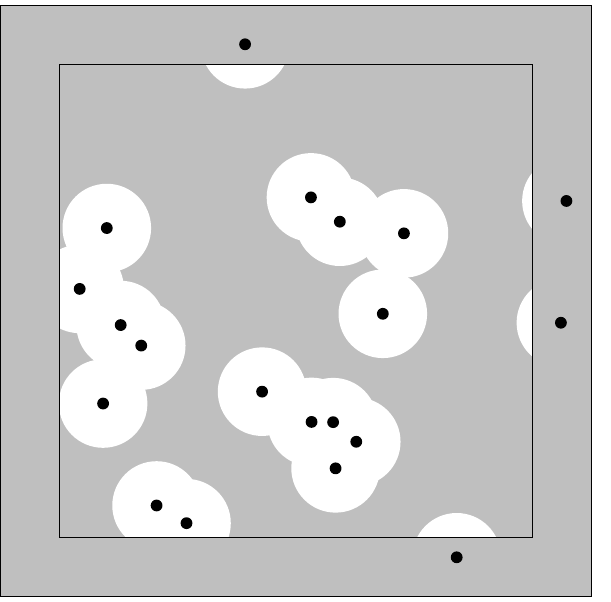}
 \end{center}
 \caption{To study the $\Gamma$-$\liminf$, we remove all solutes in the outer region of each cube to obtain $\overline{\rho}$ and we set $\overline{u}\coloneqq 1$ there. This may increase the energy only slightly. Then we cut off the electric potential to obtain a competitor $\overline{\psi}$ to the cell problem for the limit energy $\varphi(\overline{\rho}(Q))$.}
\end{figure}

We have estimated all the local terms in $E_r(\overline{\rho_r},\overline{u_r})$, leaving the nonlocal electrical energy. Since $\overline{u_r}\geq u_r$, we have
\begin{align}
& \max_\psi \int_{\Omega} Q_r \rho_r \psi - \frac{\eps(u_r)}{2r}|\nabla \psi|^2 - \frac{1}{r^3} u_r B(\psi)\,dx\\
\geq &\int_{\Omega} Q_r \rho_r \psi_r - \frac{\eps(\overline{u_r})}{2r} |\nabla \psi_r|^2 - \frac{1}{r^3}\overline{u_r} B(\psi_r)\,dx
\end{align}
for any $\psi_r\in H^1_0(\Omega)$.

We first choose a separate $\psi_{z,r}$ for every $z\in \Z^3$, namely the maximizer of
\begin{align}
 \int_{\R^3} Q_r \rho_{z,r} \psi - \frac{\eps(\overline{u_r})}{2r}|\nabla \psi|^2 - \frac{1}{r^3} \overline{u_r} B(\psi)\,dx.
\end{align}

Now we use the convexity condition \ref{B2} to get the estimate
\begin{align}
 &\int_{\R^3} \frac{\eps(\overline{u_r})}{2r}|\nabla \psi_{z,r}|^2 + \frac{1}{r^3} \overline{u_r} B(\psi_{z,r})\,dx\\
 \leq &c\int_{\R^3} Q_r \rho_{z,r} \psi_{z,r} - \frac{\eps(\overline{u_r})}{2r}|\nabla \psi_{z,r}|^2 - \frac{1}{r^3} \overline{u_r} B(\psi_{z,r})\,dx.
\end{align}

We can now cut off each $\psi_{z,r}$ to get a function $\overline{\psi_{z,r}}$ in $H^1_0(Q_{z,z_0,r^{1/3},r^{1/3}})$, which we glue together. As in the upper bound, we use the fact that $\overline{u_r}=1$ in $Q_{z,z_0,r^{1/3},r^{1/3}} \setminus Q_{z,z_0,r^{1/3},r^{1/3}, r^{2/3}/2}$, and choose $N_{z,r}\in \{1,\ldots,\lfloor r^{-1/3}/2\rfloor\}$ such that
\begin{align}
 \int_{Q_{z,z_0,r^{1/3},r^{1/3}-N_{z,r}r}\setminus Q_{z,z_0,r^{1/3},r^{1/3}-(N_{z,r}+1)r}} \frac{\eps_1}{2r}|\nabla \psi_{z,r}|^2 + \frac{1}{r^3}  B(\psi_{z,r})\,dx\\
 \leq 2r^{1/3} \int_{Q_{z,z_0,r^{1/3},r^{1/3}} \setminus Q_{z,z_0,r^{1/3},r^{1/3}, r^{2/3}/2}}  \frac{\eps_1}{2r}|\nabla \psi_{z,r}|^2 + \frac{1}{r^3}  B(\psi_{z,r})\,dx.
\end{align}

Take $\eta_{z,r}\in C_c^\infty(Q_{z,z_0,r^{1/3},r^{1/3}-N_{z,r}r})$ to be $1$ in $Q_{z,z_0,r^{1/3},r^{1/3}-(N_{z,r}+1)r}$ and with $\|\nabla \eta_{z,r}\|_{L^\infty}\leq 2/r$. Setting $\overline{\psi_{z,r}} \coloneqq \psi_{z,r}\eta_{z,r}$, we obtain
\begin{align}
 \int_{\R^3} \frac{\eps(\overline{u_r})}{2r}|\nabla \overline{\psi_{z,r}}|^2 + \frac{1}{r^3} \overline{u_r} B(\overline{\psi_{z,r}})\,dx \leq (1+16r^{1/3})\int_{\R^3} \frac{\eps(\overline{u_r})}{2r}|\nabla \psi_{z,r}|^2 + \frac{1}{r^3} \overline{u_r} B(\psi_{z,r})\,dx.
\end{align}

Since $Q_r\rho_{z,r}$ is supported where $\overline{\psi_{z,r}}=\psi_{z,r}$, the term $\int_{\R^3} Q_r\rho_{z,r} \psi_{z,r}\,dx$ remains unchanged.

We define the global $\overline{\psi_r} \coloneqq \sum_{z\in \Z^3\,:\,Q_{z,r}\subset \Omega} \overline{\psi_{z,r}}\in H^1_0(\Omega)$, which satisfies
\begin{align}\label{H^1 estimate modified potential}
 \int_{\Omega}\frac{\eps(\overline{u_r})}{2r}|\nabla \overline{\psi_r}|^2 + \frac{1}{r^3}\overline{u_r}B(\overline{\psi_r})\,dx
 \leq (1+ 16r^{1/3}) \sum_{z\in \Z^3} \int_{\R^3} \frac{\eps(\overline{u_r})}{2r}|\nabla \psi_{z,r}|^2 + \frac{1}{r^3} \overline{u_r} B(\psi_{z,r})\,dx.
\end{align}

By \eqref{H^1 estimate modified potential}, \eqref{LJ estimate modified}, \eqref{TV estimate modified} we see that
\begin{align}\label{eq: estimateoalpha}
 &E_r(\overline{\rho_r},\overline{u_r},\overline{\psi_r}) 
 \leq E_r(\rho_r,u_r,\overline{\psi_r}) + o(\alpha(r)) + \int_{\Omega}Q_r(\overline{\rho_r} - \rho_r)\overline{\psi_r}\,dx\\
 %\leq &\max_\psi E_r(\rho_r,u_r,\psi) + o(\alpha(r)) + c\|Q_r(\overline{\rho_r} -\rho_r)\|_{L^\infty}^{1/2} \left(|\rho_r - \overline{\rho_r}|(\Omega)\right)^{1/2}\|\overline{\psi_r}\|_{L^2(A_{z_0,r})}.
\end{align}
We claim that $\int_{\Omega} Q_r(\overline{\rho_r} - \rho_r)\overline{\psi_r}\,dx = o(\alpha(r))$. 
By \ref{B1}, find $K>0$ such that $B(x) \geq \frac12 |x|^2$ for all $|x| \geq K$.
Define $\overline{\psi_r}^b \coloneqq \overline{\psi_r} \1_{|\overline{\psi_r}| \leq K}$ which is clearly bounded by $K$. 
Moreover, $B(\overline{\psi_r} - \overline{\psi_r}^b) \geq |\overline{\psi_r} - \overline{\psi_r}^b|^2$. 
For the bounded part of $\overline{\psi_r}$ we find the simple estimate
\[
 \int_{\Omega} Q_r(\overline{\rho_r} - \rho_r)\overline{\psi_r}^b\,dx \leq C K |\overline{\rho_r} - \rho_r|(\Omega) = C K |\rho_r|(A_{z_0,r}) \leq C r^{\frac13} \alpha(r).
\]
Hence, it remains to control $ \int_{\Omega} Q_r(\overline{\rho_r} - \rho_r)(\overline{\psi_r} - \overline{\psi_r}^b)\,dx$.
Note that 
\[
 \int_{\Omega} Q_r(\overline{\rho_r} - \rho_r)(\overline{\psi_r} - \overline{\psi_r}^b)\,dx \leq C \|Q_r(\overline{\rho_r} -\rho_r)\|_{L^\infty}^{1/2} \left(|\rho_r - \overline{\rho_r}|(\Omega)\right)^{1/2}\|\overline{\psi_r} - \overline{\psi_r}^b\|_{L^2(A_{z_0,r})}. 
\]
Here we use the technical assumption of a solute concentration bound in the definition of $\mathcal{A}_r$ (see \eqref{def: Ar}) to obtain
\begin{align}
\|Q_r(\overline{\rho_r} -\rho_r)\|_{L^\infty} \leq CM/r^3.
\end{align}
Again, by construction, we also have that $|\overline{\rho_r} - \rho_r|(\Omega) = |\rho_r|(A_{z_0,r}) \leq cr^{1/3} \alpha(r)$.
To estimate the $L^2$-norm of $\overline{\psi_r} - \overline{\psi_r}^b$ on $A_{z_0,r}$, note first that wherever $\overline{u_r}$ is $1$, we can simply estimate $|\overline{\psi_r} - \overline{\psi_r}^b|^2 \leq u B(\overline{\psi_r} - \overline{\psi_r}^b)$. 
The only parts of $A_{z_o,r}$ where $\overline{u_r}$ might not be $1$, are located in a neighborhood of thickness $Lr$ around the boundaries of the $q_{z,r}$.
Here, we can use Poincar\'e's inequality to find altogether
\begin{align}
 \|\overline{\psi_r} - \overline{\psi_r}^b\|_{L^2(A_{z_0,r})} \leq & Cr^{3/2}\sqrt{\frac{1}{r^3}\int_{\Omega}\overline{u_r}B(\overline{\psi_r}- \overline{\psi_r}^b)\,dx + \int_{\Omega} \frac{\eps(\overline{u_r})}{2r} \left|\nabla \left(\overline{\psi_r} - \overline{\psi_r}^b\right)\right|^2 \, dx } \\
 \leq & Cr^{3/2}\sqrt{\alpha(r)},
\end{align}
and consequently
\begin{align}
 \|Q_r(\overline{\rho_r} -\rho_r)\|_{L^\infty}^{1/2} \left(|\rho_r - \overline{\rho_r}|(\Omega)\right)^{1/2}\|\overline{\psi_r} - \overline{\psi_r}^b\|_{L^2(A_{z_0,r})} \leq C r^{1/6}\alpha(r).
\end{align}
This proves the claim.

Hence, the energy $E_r(\overline{\rho_r},\overline{u_r},\overline{\psi_r})$ can be localized to the different cubes $Q_{z,z_0,r^{1/3},r^{1/3}}$, where $\overline{\psi_r}=\overline{\psi_{z,r}}$ has not much less energy than the maximizer $\psi_{z,r}$ by \eqref{H^1 estimate modified potential}. 
We have
\begin{align}
&\sum_{z\in\Z^3\,:\,Q_{z,r}\subset \Omega} \varphi(\rho_{z,r}(\Omega))\\
\leq &\sum_{z\in \Z^3\,:\,Q_{z,r}\subset \Omega} E_r(\rho_{z,r},\overline{u_r},\psi_{z,r})\\
 \leq & E_r(\overline{\rho_r},\overline{u_r},\overline{\psi_r}) + 16c r^{1/3} \sum_{z\in \Z^3\,:\,Q_{z,r}\subset \Omega} \int_{\R^3} \left(\frac{\eps(\overline{u_r})}{2r}|\nabla \psi_{z,r}|^2 + \frac{1}{r^3} \overline{u_r} B(\psi_{z,r})\right)\,dx\\
 \leq & \max_\psi E_r(\rho_r,u_r,\psi) + o(\alpha(r)) + 16 c^2 r^{1/3} \sum_{z\in \Z^3\,:\,Q_{z,r}\subset \Omega}\int_{\R^3} \left(Q_r\overline{\rho_{z,r}} \psi_{z,r} -  \frac{\eps(\overline{u_r})}{2r}|\nabla \psi_{z,r}|^2 - \frac{1}{r^3} \overline{u_r} B(\psi_{z,r}) \right)\,dx,
\end{align}
where in the last inequality we used \eqref{eq: estimateoalpha}, and once again the conditions \ref{B1} and \ref{B2}. Now
\begin{align}
  \sum_{z\in \Z^3\,:\,Q_{z,r}\subset \Omega}\int_{\R^3} \left(Q_r\overline{\rho_{z,r}} \psi_{z,r} -  \frac{\eps(\overline{u_r})}{2r}|\nabla \psi_{z,r}|^2 - \frac{1}{r^3} \overline{u_r} B(\psi_{z,r})\right)\,dx \leq C |\rho|(\Omega) \leq C \alpha(r).
\end{align}
The sum of the $\varphi(\rho_{z,r}(\Omega))$ can be written as
\begin{align}
 \sum_{z\in\Z^3\,:\,Q_{z,r}\subset \Omega} \varphi(\rho_{z,r}(\Omega))
=  E_0(\overline{\rho_r}),
\end{align}
where $\overline{\rho_r} \coloneqq \sum_{z \in \Z^3\,:\,Q_{z,r}\subset \Omega} \rho_{z,r}(\Omega) \delta_{z+z_0}$.

We have $\overline{\rho_r}/\alpha(r)\weakstar \rho$ since $\rho_r/\alpha(r) \weakstar \rho$. 
Since $E_0$ is positively $1$-homogeneous and vaguely lower semi-continuous (see Lemma \ref{lemma: subadditivity}), we conclude that
\begin{align}
 E_0(\rho) \leq & \liminf_{r\to 0} \frac{E_0(\overline{\rho_r})}{\alpha(r)}\\
 \leq & \liminf_{r\to 0} \frac{\max_\psi E_r(\rho_r,u_r,\psi)+o(\alpha(r))}{\alpha(r)}\\
 = & \liminf_{r\to 0} \frac{ E_r(\rho_r,u_r)}{\alpha(r)}.
\end{align}
 \end{proof}

\section{The Case $B=0$}\label{sec: noB}

In this section, we consider the energy without ionic effect i.e., $B=0$. 
Unlike in the case with an ionic effect, in this section the electrical energy consists of two competing effects, the self-energy of each solute ion and the electrical interaction of the different solutes. 
Both occur on different scales which leads to three different scaling regimes depending on the scaling law of $r$ and the number of solutes $\alpha(r)$ (see Subsection \ref{sec: heuristics}).
% \\
% Over the whole section we assume
% \begin{multicols}{2}
% \begin{enumerate}
% \item[i)] $r \to 0$, 
% \item[ii)] $\alpha(r) \to \infty$,
% \end{enumerate}
% \columnbreak
% \begin{enumerate}
% \item[iii)] $\delta_n \gg r$,
% \item[iv)] $\alpha(r) \delta_n^3 \to 0$.
% \end{enumerate}
% \end{multicols}

\subsection{The Subcritical Regime}\label{sec: subcriticalB}

In this section we assume that $\alpha(r) r \to 0$.
By the heuristics discussed in Subsection \ref{sec: heuristics} the self-energy dominates the interaction energy in this regime.
Indeed, in this section we prove Theorem \ref{thm: GammanoBsub}, i.e, we show that the rescaled energy $\frac{E_{r}}{\alpha(r)} $ where $E_{r}$ is defined as in \eqref{energy definition} $\Gamma$-converges with respect to vague convergence of $\frac{\rho_r}{\alpha(r)}$ in $\mathcal{M}(\Omega;\R^N)$ to the energy $E^{sub}: \mathcal{M}(\Omega,\R^N) \rightarrow [0,\infty]$ defined by 
\[
E^{sub}(\rho) = \begin{cases} \int_{\Omega} \varphi\left( \frac{d \rho}{d|\rho|} \right) \,d|\rho| &\text{ if each component of the vector-valued measure } \rho \text{ is nonnegative }, \\ +\infty &\text{ otherwise.}\end{cases}
\]
The proof is split into Propositions \ref{prop: subcomp}, \ref{prop: sublower}, and \ref{prop: subupper}.\\
We start with the compactness result which is immediate as the energy $E_r(\rho_r)$ contains the total variation of the measure $\rho_r$.

\begin{proposition}[Compactness]\label{prop: subcomp}
Let $r \to 0$ and $\alpha(r) \to \infty$ such that $\alpha(r) r \to 0$. 
Assume that $a >0$ is so large that $E_r(\rho) \geq c |\rho|(\Omega)$ for some $c>0$.
Let $\{\rho_r\}_{r>0}$ be a sequence of measures such that $E_{r}(\rho_r) \leq C \alpha(r)$ for a universal constant $C > 0$.
 Then there exists a measure $\rho \in \mathcal{M}(\Omega;\R^N)$ such that $\rho^i$ is a nonnegative measure for all $i=1,\dots,N$ and up to a subsequence it holds
 \[
  \frac{ \rho_r }{\alpha(r)} \weakstar\rho \text{ in } \mathcal{M}(\Omega). 
 \]
\end{proposition}
\begin{proof}
 By the assumptions it follows directly that $|\rho_r|(\Omega) \lesssim \alpha(r)$. 
 Then the existence of the convergent subsequence is classical.
\end{proof}
\begin{remark} 
 If one assumes that all solute molecules are separated at scale $r$ then one can use an argument similar to the heuristics to prove compactness even for $a=0$ provided that the integral of
 each Lennard-Jones potential is not too negative.
\end{remark}

Next, we show the $\liminf$-inequality.
We start with a lemma which allows us to find a clustering of the solutes whose interaction is negligible as $r \to 0$.

\begin{lemma}\label{lemma: cubes}
Let $L>0$, $u_r: \Omega \to \{0,1\}$, $\rho_r = \sum_{i=1}^N \sum_{j=1}^{m^i} e_i \delta_{x^{i,j}}$ for points $x^{i,j} \in \Omega$
, and $M= \sum_{i=1}^N m^i$. 
Then there exists $C>0$ such that for all $\delta > 2Lr$ there exist disjoint half-open cubes $(Q_k)_k$ of sidelength $\delta$ and cubes $(\tilde{Q}_k)_k$ with the same centers and sidelength $\delta + 2 Lr$ such that 
 \begin{enumerate}
 \item $\supp(\rho_r) \subseteq \bigcup_k Q_k$,
 \item $\sum_k \sum_{x^{i,j} \in Q_k, x^{i',j'} \notin Q_k} \frac{1}{|x^{i,j}-x^{i',j'}|} \leq C \frac{M^2}{\delta}$,
 \item $\sum_i \int_{(\tilde{Q}_i \setminus Q_i) \cap \Omega} r^{-2} \beta |Du_r| + r^{-3} \gamma (1-u_r) \, dx \leq C \frac{Lr}{\delta} \int_{\Omega} r^{-2} \beta |Du_r| + r^{-3} \gamma  (1-u_r) \, dx.$
 \end{enumerate}
\end{lemma}
\begin{proof}
 For $z \in \R^3$ we define $Q_{z,\delta} \coloneqq z + [-\delta/2,\delta/2)^3$. 
 Notice that given $x,y\in \Omega$, if $x \in Q_{z,\delta}$ and $y \notin Q_{z,\delta}$, then $z \in Q_{x,\delta}$ and at least one component $z_j$ of $z$ lies in an interval of length less than $|x_j-y_j| \leq |x-y|$. 
 Hence, the measure of all $z$'s having this property can be estimated by $3 \min\{|x-y|,\delta\} \delta^2$.  
 Therefore,
 \begin{align*}
  &\int_{[0,\delta)^3} \sum_{k \in \delta \Z^3} \sum_{x^{i,j} \in Q_{k + z,\delta}, x^{i',j'} \notin Q_{k+z,\delta}} \frac{1}{|x^{i,j}-x^{i',j'}|} \, dz \\ 
  \leq &3 \sum_{x^{i,j} , x^{i',j'}} \frac{\min\{\delta,|x^{i,j}-x^{i',j'}|\} \delta^2}{|x^{i,j}-x^{i',j'}|} \\
  \leq &3 M^2 \, \delta^2.
 \end{align*}
 Therefore, for a subset of $[0,\delta)^3$ of measure $\frac{3}{4} \delta^3$ we have the inequality
 \begin{equation}\label{eq: estlemma1}
 \sum_k \sum_{x^{i,j} \in Q_k, x^{i',j'} \notin Q_k} \frac{1}{|x^{i,j}-x^{i',j'}|} \leq  4 \frac{M^2}{\delta}.
 \end{equation}
Next, note that a similar argument shows that
\begin{align*}
&\int_0^{\delta} \sum_{k\in \Z}\int_{ \left((s-\delta/2 + \delta k - Lr,s-\delta/2 +\delta k) \times \R^2 \cup (s+\delta/2+\delta k,s+\delta/2+\delta k+Lr) \times \R^2 \right) \cap \Omega} r^{-2}\beta |Du_r| + r^{-3} \gamma  (1-u_r) \, dx \,ds \\
\leq 
&2 Lr \int_{\Omega} r^{-2}\beta |Du_r| + r^{-3} \gamma  (1-u_r) \, dx.
\end{align*}
Repeating the same argument in the other two cardinal directions and combining shows that for a subset of $[0,\delta)^3$ of measure $\frac34 \delta^3$ it holds
\begin{equation}\label{eq: estlemma2}
\int_{(Q_{k+z,\delta+2Lr} \setminus Q_{k+z,\delta}) \cap \Omega} r^{-2}\beta |Du_r| + r^{-3} \gamma  (1-u_r) \, dx \leq C \frac{L_r}{\delta} \int_{\Omega} r^{-2}\beta |Du_r| + r^{-3} \gamma  (1-u_r) \, dx
\end{equation}
In particular, there exist $z\in[0,\delta)^3$ such that \eqref{eq: estlemma1} and \eqref{eq: estlemma2} are satisfied. 
\end{proof}

Now we are ready to show the $\liminf$-inequality.

\begin{proposition} \label{prop: sublower}
 Let $r \to 0$.
 Let $\alpha(r) \to \infty$ such that $\alpha(r) r \to 0$.
 Assume that $\frac{\rho_r}{\alpha(r)} \weakstar \rho$ in $\mathcal{M}(\Omega;\R^N)$. 
Then 
\[
  \liminf_{r \to 0} \frac1{\alpha(r)}E_{r}(\rho_r) \geq \int_{\Omega} \varphi\left( \frac{d\rho}{d|\rho|} \right) \,d|\rho|.
\]
\end{proposition}
\begin{proof}
First, we may assume that $\liminf_{r \to 0} \frac{1}{\alpha(r)} E_{r}(\rho_r) = \lim_{r \to 0} E_{r}(\rho_r)$ and $\sup \frac{1}{\alpha(r)} E_{r}(\rho_r) < \infty$. 
Then, by the compactness result, we know that each component of $\rho$ is a nonnegative.
Moreover, it follows that $\rho_r \in \mathcal{A}_{r}(\Omega)$.
We write $\rho_r = \sum_{i=1}^N \sum_{j=1}^{m_r^i} e_i \delta_{x_r^{i,j}}$. 
%Now, we define the set $A_r = \{x_r^{i,j}: i \in\{1,\dots,N\}, j \in \{1,\dots,m_r^i\} \}$.
By the convergence of $\frac{\rho_r}{\alpha(r)}$ it holds that $\sum_{i=1}^N m_r^i \leq C \alpha(r)$. 
Moreover, by Lemma \ref{minimax lemma} let $u_r: \Omega \rightarrow \{0,1\}$ such that $E_r(\rho_r) = E_r(\rho_r,u_r)$ which we extend by $1$ to $\R^3$.\\
Next, let let $\delta_r \searrow 0$ such that $\frac{\alpha(r) r}{\delta_r} \rightarrow 0$ and $\frac{\delta_r}{r} \to \infty$, and $L>1$ such that $U_{r,\rho_r}(x) \leq 0$ for all $x \notin B_{Lr}(\supp( \rho_r))$.
Applying Lemma \ref{lemma: cubes} for $L$ and $\delta_r$ yields the existence of disjoint cubes $(Q_k^r)_k$ of sidelength $\delta_r$ and cubes $\tilde{Q}_k^r$ of sidelength $\delta_r + 2Lr$ such that $\supp(\rho_r) \subseteq \bigcup_k Q_k^r$ and
\begin{enumerate}[label=(\roman*)]
\item $\sum_k \sum_{x^{i,j_r} \in Q_k, x^{i',j'}_r \notin Q_k} \frac{1}{|x^{i,j}-x^{i',j'}|} \leq C \frac{\alpha(r)^2}{\delta_r}$, \label{eq: lowerinteraction}
 \item $\sum_i \int_{(\tilde{Q}_i \setminus Q_i) \cap \Omega} r^{-2} \beta |Du_r| + r^{-3} \gamma (1-u_r) \, dx \leq C \frac{Lr}{\delta_r} \int_{\Omega} r^{-2} \beta |Du_r| + r^{-3} \gamma  (1-u_r) \, dx.$
\end{enumerate}
Define $\mathcal{G} = \{ Q_k^r: \operatorname{dist}(Q^r_k,\partial\Omega) \geq \delta_r\}$ and
\[ 
A_r = \bigcup_{Q_k^r \in \mathcal{G}} Q_k^r.
\]
Then, for $i,j$ such that $x_r^{i,j} \in A_r$ let $\psi_r^{i,j}$ be the solution to $-\operatorname{div }\left( \frac{\eps(u_r)}{r} \nabla \psi \right) = Q_r (e_i \delta_{x_r^{i,j}})$. 
By \cite{Littman1963} we can write 
\[
 \psi_r^{i,j}(y) = \int_{\R^3} Q_r (e_i \delta_{x_r^{i,j}})(y) K^{i,j}_r(x,y) \, dy,
\]
where $K_r^{i,j}$ is the fundamental solution for the translated $\eps(u_r)$ and satisfies $|K_r^{i,j}| \leq C \frac{r}{|x-x_r^{i,j}|}$ for a constant which does neither depend on $r$ nor $i,j$.
In particular, we find from the $L^{\infty}$-bounds and compact support of $\phi_i$ that 
\begin{equation} \label{eq: solupper}
|\psi_r^{i,j}(x)| \leq C r \int_{\R^3} r^{-3} |\phi_i|\left(\frac{y-x_r^{i,j}}{r^3}\right) \frac{1}{|x_r^{i,j}+x-y|} \, dy  \lesssim C r \fint_{B_{Lr}(x_r^{i,j})} \frac{1}{|x-y|} \,dy \lesssim C\frac{r}{|x-x_r^{i,j}|}.
\end{equation}
Using this bound, a similar argument shows that for all $x_r^{i,j}$ and $x_r^{l,m} \in A_r$ it holds that
\[
 \int_{\R^3} |Q_r (e_i \delta_{x_r^{i,j}})(y)| |\psi_r^{l,m}|(y) \, dy \lesssim \frac{r}{|x_r^{i,j} - x_r^{l,m}|}.
\]
Hence, we derive from \ref{eq: lowerinteraction} that
\begin{equation} \label{eq: upperboundinter}
 \sum_{Q_k^r} \sum_{x_r^{i,j} \in Q_k^r} \sum_{x_r^{l,m} \in A_r, x_r^{l,m} \notin Q_k^r} \int_{\R^3} |Q_r (e_i \delta_{x_r^{i,j}})(y)| |\psi_r^{l,m}|(y) \, dy \leq C \frac{\alpha(r)^2}{\delta_r} \rightarrow 0 \text{ as } r \to 0.
\end{equation}
We define 
\[
\psi^k_r \coloneqq \sum_{x_r^{i,j} \in Q_k^r} \psi_r^{i,j} \text{ and } \psi_r \coloneqq \sum_{Q_k^r \in \mathcal{G}} \psi_r^{k} \coloneqq \sum_{x_r^{i,j} \in A_r} \psi_r^{i,j} .  
\]
Then \eqref{eq: solupper} and the definition of $A_r$ yield that for all $x \in \partial{\Omega}$ it holds that $|\psi_r(x)| \leq C \frac{\alpha(r) r}{\delta_r} \rightarrow 0$ as $r\to 0$.
Hence, if we define $R_r \in H^1(\Omega)$ to be the solution to $-\operatorname{div}\left(\frac{\eps(u_r)}{r} \nabla R_r\right) = 0$ and $R_r(x) = - \psi_r(x)$ for all $x \in \partial \Omega$
it follows by the maximum principle that $R_r \rightarrow 0$ uniformly in $\Omega$. \\
Integration by parts shows for the electrical energy that
\begin{align*}
 E_r^{el}(\rho_r,u_r) \geq &\int_{\Omega} Q_r \rho_r (\psi_r + R_r) - \frac{\eps(u_r)}{2r} |\nabla (\psi_r + R_r)|^2 \, dx \\
 = &\frac12 \int_{\Omega} Q_r \rho_r (\psi_r + R_r)dx + \frac12 \int_{\Omega} \sum_{x^{i,j}_r \notin A_r} Q_r(e_i \delta_{x_r^{i,j}}) (\psi_r + R_r) \, dx \\
 = &\frac12 \int_{\Omega} (Q_r \rho_r + \sum_{x^{i,j}_r \notin A_r} Q_r(e_i \delta_{x_r^{i,j}})) R_r \, dx + 
 \frac12 \int_{\Omega} \sum_{Q_k^r \in \mathcal{G}} Q_r((\rho_r)_{|Q_k^r}) \psi_r^{k} \, dx \\
 &+ \frac12 \int_{\Omega} \sum_{Q_k^r \in \mathcal{G}} \sum_{x_r^{i,j} \in Q_k^r, x_r^{l,m} \notin Q_k^r} Q_r(e_i \delta_{x_r^{i,j}}) \psi_r^{l,m} \,dx \\
 &+ \frac12 \int_{\Omega} \sum_{Q_k^r} \sum_{x_r^{i,j} \in Q_k^r} \sum{x_r^{l,m} \notin Q_k^r, x_r^{l,m} \in A_r} Q_r(e_i \delta_{x_r^{i,j}}) \psi_r^{l,m}\, dx 
 \end{align*}
First note for the first term divided by $\alpha(r)$ goes to zero since the occurring measure divided by $\alpha(r)$ is bounded in total variation and $R_r$ goes to zero uniformly.
By \eqref{eq: upperboundinter} also the third and fourth term divided by $\alpha(r)$ converges to zero as $r\to0$.
Hence, 
\begin{align*}
 E_r^{el}(\rho_r,u_r) \geq &\frac12 \int_{\Omega} \sum_{Q_k^r \in \mathcal{G}} Q_r((\rho_r)_{|Q_k^r}) \psi_r^{k} \, dx + o(\alpha(r))  \\
 =  &\sum_{Q_k^r \in \mathcal{G}} \int_{\R^3} Q_r((\rho_r)_{|Q_k^r}) \psi_r^{k}  - \frac{\eps(u_r)}{2} |\nabla \psi_r^k|^2 \, dx + o(\alpha(r)) \\
 = &\sum_{Q_k^r \in \mathcal{G}} E^{el}_r((\rho_r)_{|Q_k^r},u_r;\R^3) + o(\alpha(r)).
\end{align*}
Assuming that $a>0$ is large enough such that the total variation of the measure outside $\bigcup_{Q_k^r \in \mathcal{G}} Q_k^r$ compensates the potential negativeness of the corresponding Lennard-Jones term this yields for the full energy
\begin{align}
 E(\rho_r) \geq   &o(\alpha(r)) + \sum_{Q_k^r \in \mathcal{G}} a |\rho_r|(Q_k^r) + r^{-3} \beta \int_{Q_k^r} (1-u_r) \,dx + r^{-2} \gamma |D u_r|(Q_k^r) + r^{-3} \int_{\R^3} U_{r,(\rho_r)_{|Q_k^r}}(x) u_r(x) \,dx \\
 &+E^{el}_r((\rho_r)_{|Q_k^r},u_r;\R^3).
\end{align}
Finally, let us fix a cube $Q_k^r \in \mathcal{G}$.
We define $\tilde{u}_r^k: \R^3 \to \{0,1\}$ by
\[
 \tilde{u}_r^k(x) = \begin{cases}
                     u_r(x) &\text{if } x \in \tilde{Q}_k^r, \\ 1 &\text{else},
                    \end{cases}
\]
where $\tilde{Q}_k^r$ is the cube as constructed above with sidelength $\delta_r + 2 Lr$ and the same center as $Q_k^r$ where $L>1$ was defined to be so that $U_{r,(\rho_r)_{|Q_k^r}} \leq 0$ outside 
$\tilde{Q}_k^r$.
We see immediately that
\begin{align*}
 &r^{-3} \beta \int_{\R^3} (1-\tilde{u}^k_r) \,dx + r^{-3} \int_{\R^3} U_{r,(\rho_r)_{|Q_k^r}}(x) \tilde{u}^k_r(x) \,dx +  E^{el}_r((\rho_r)_{|Q_k^r},\tilde{u}^k_r;\R^3) \\
 \leq &r^{-3} \beta \int_{\tilde{Q}^k_r} (1-u_r) \,dx + r^{-3} \int_{\R^3} U_{r,(\rho_r)_{|Q_k^r}}(x) u_r(x) \,dx +  E^{el}_r((\rho_r)_{|Q_k^r},u_r;\R^3).
\end{align*}
Moreover, we can argue as in \eqref{TV estimate modified} in the proof of the lower bound in the case where $B$ satisfies \ref{B1} and \ref{B2} that
\[
 |D\tilde{u}_r^k|(\R^3) \leq |D u_r|(\tilde{Q}_k^r) + \frac Cr \int_{\tilde{Q}_k^r \setminus Q_k^r} (1 - u_r) \, dx.
\]
Hence, we find since $\sum_{Q_k^r \in \mathcal{G}} |Du_r|(\tilde{Q}_k^r \setminus Q_k^r) + \int_{\tilde{Q}_k^r \setminus Q_k^r} (1-u_r) \, dx \lesssim \frac{Lr}{\delta_r} \alpha(r)$ that  
\begin{equation}\label{eq: lowerfinalnoB}
 E_r(\rho_r) \geq   o(\alpha(r)) + E_r((\rho_r)_{|Q_k^r}, \tilde{u}_r^k;\R^3) \geq o(\alpha(r)) + E_r((\rho_r)_{|Q_k^r};\R^3).
\end{equation}
Define $\tilde{\rho}_r = \sum_{Q_k^r \in \mathcal{G}} \rho(Q_k^r) \delta_{x_k^r}$ where $x_k^r$  is the center of $Q_k^r$.
Then also $\frac{\tilde{\rho}_r}{\alpha(r)} \weakstar \rho$ in $\mathcal{M}(\Omega;\R^N)$.
Moreover, by definition of the self-energy $\varphi$, Lemma \ref{lemma: subadditivity}, and Reshetnyak's theorem it follows from \eqref{eq: lowerfinalnoB} that 
\begin{align*}
\liminf_{r\to 0} \frac1{\alpha(r)} E_r(\rho_r) &\geq \liminf_{r\to0} \int_{\Omega} \varphi\left(\frac{d\tilde{\rho}_r}{d|\tilde{\rho}_r|} \right) \, d|\tilde{\rho}_r| + o(1). \\
&\geq \int_{\Omega} \varphi\left(\frac{d\rho}{d|\rho|} \right) \, d|\rho|.
\end{align*}

\end{proof}

Finally, we prove the existence of a recovery sequence for the energy $E^{sub}$.
The construction is very similar to the case with a superquadratic $B$ as the constructed approximating sequence is dilute enough to neglect the interactions of the different occurring electrical fields.

\begin{proposition}\label{prop: subupper}
 Let $r \to 0$ and $\alpha(r) \to \infty$ such that $\alpha(r) r \to 0$. 
Let $\rho \in \mathcal{M}(\Omega;\mathbb{R}^N)$. 
 Then there exists a sequence of measures $\{\rho_r\}_{r>0} \subseteq \mathcal{M}(\Omega;\mathbb{R}^N)$ such that $\frac{\rho_r}{\alpha(r)} \weakstar \rho$ in $\mathcal{M}(\Omega;\R^N)$ and
 \[
  \limsup_{r \to 0} \frac1{\alpha(r)}E_{r}(\rho_r) \leq \int_{\Omega} \varphi\left( \frac{d\rho}{d|\rho|}\right)  \,d|\rho|.
 \]
\end{proposition}
\begin{proof}
We divide the proof in three steps. 
For simplicity, we first prove the existence of a recovery sequence for a single a Dirac measure. 
This construction is then easily used for sums of Dirac measures. 
Finally, we finish the proof by a classical energy density argument which is also recalled in the case $B\neq0$.\\
\begin{figure}
\begin{center}

\includegraphics{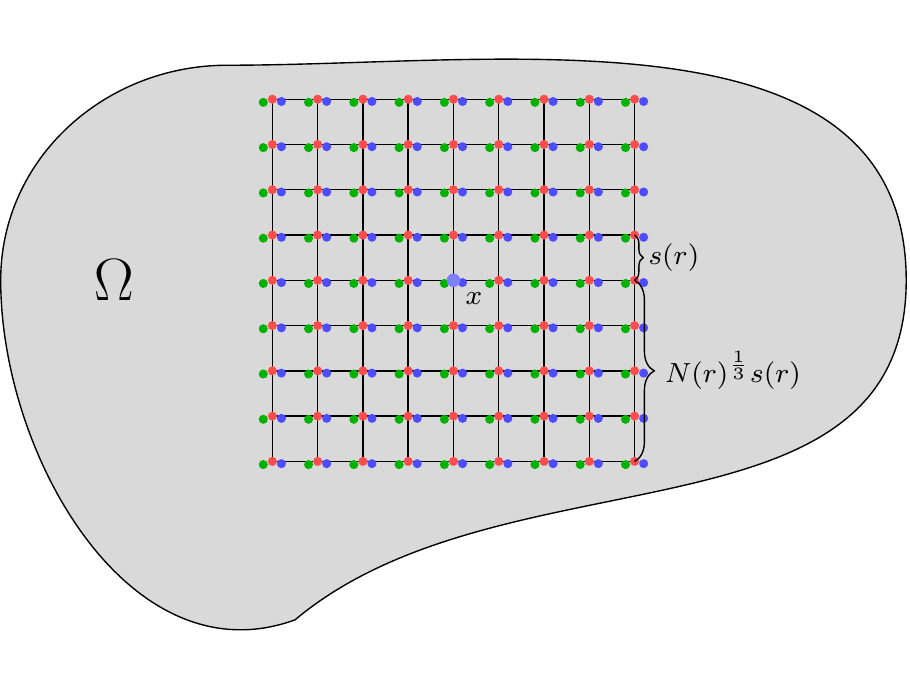}

\caption{Sketch of the construction for $\rho_r$. Each collection of red, blue, and green dots represents a translated rescaled optimal measure $\mu$.} 
\end{center}
\label{figure: upperboundsub}
\end{figure}
 \underline{\textbf{Step 1:}} $\rho = \xi \delta_x$ for some $\xi \in \R^N, \xi_i\geq 0$ for all $i=1,\dots,N$ and $x\in \Omega$. \\
 Let $\delta>0$.
First, by the definition of $\varphi$ there exists $\mu \in \mathcal{A}_1(\Omega)$ and $z >0$ such that 
\begin{align}
 \left|\frac{\mu(\mathbb{R}^3)}{z} - \xi \right| \leq \delta \text{ and } \frac{E_1(\mu)}{z} - \varphi(\xi) \leq \delta. 
\end{align}
Next, let $R>0$ such that $\mu$ is supported in $B_R(0)$ and it holds
\begin{align}
 \int_{\R^3 \setminus B_R(0)} U_{1,\mu} \,dx \geq - \delta. 
\end{align}
We define $N(r) = \lfloor \frac{\alpha(r)}{z}\rfloor$ and the lattice spacing $s(r) = r^{\frac13}$.
It follows that $\frac{s(r)}{r} \to \infty$ and $N(r) s(r)^3 \to 0$ as $r\to0$.
The competitor $\rho_r$ is then defined by
\[
 \rho_r = \sum_{l \in Z_{N(r)}} \rho_r^l = \sum_{l \in Z_{N(r)}} \mu\left( \frac{\cdot - s(r) l-x}{r} \right) 
\]
where $Z_{N} \subseteq \Z^3$ is a set of size N with $\max_{l \in Z_{N}} |l| \leq N^{\frac13}$. \\
Then, one can check that $\rho_r \in \mathcal{A}_r(\Omega)$ for $r>0$ small enough and as $r\to 0$ we find that $\frac{\rho_r}{ \alpha(r)} \weakstar \frac{\mu(\R^3)}{z}\, \delta_x$ which is close to $\rho$ in total variation.
Next, let $(u,\psi)$ be a minimax pair for $E_1(\mu;\R^3)$. 
Note that (by possibly enlarging $R$) we may assume that $u=1$ outside $B_R(0)$.
We define $u_r: \Omega \rightarrow \{0,1\}$ to be 
\[
 u_r(y) = \begin{cases} 
         u\left( \frac{y - s(r) l-x}{r} \right) &\text{ if } y \in B_{rR}(x + l s(r)) \text{ for some } l \in Z_{N(r)}, \\ 1 &\text{ else}.
        \end{cases}
\]
Using a similar argument as in the lower bound for $B \neq 0$, \eqref{TV estimate modified} for the total variation term, one shows for the local terms of the energy $E_r(\rho_r,u_r)$ that 
\begin{align}
&\frac1{\alpha(r)} \frac{1}{r^3}\int_{\Omega} u_r U_{r,\rho_r}\,dx + \frac{1}{r^2}\gamma |Du_r|(\Omega) + \frac{1}{r^3}\beta \int_{\Omega} (1-u_r)\,dx + a |\rho_r|(\Omega) \\
 \leq &\frac{N(r)}{\alpha(r)}\left( \int_{\R^3} u U_{1,\mu}\,dx + \gamma |Du|(\R^3) + \beta \int_{\R^3} (1-u)\,dx + a |\mu|(\R^3) + \delta + o(1) \right) \label{eq: upestlocal}
\end{align}
Note that we used that the Lennard-Jones term is non-positive outside the balls with radius $Rr$. \\
Next, we consider the electrical energy.
For $l \in Z_{N(r)}$ let $\psi^l_r$ be the unique solution to $-\operatorname{div } \frac{\eps(u_r)}{2 r} \nabla \psi^l_r = Q_r \mu\left( \frac{\cdot - s(r) l}{r} - x \right)$ in $\R^3$ such that $\psi^l \to 0$ as $|x| \to \infty$.
Here, we extend $u_r$ by $1$ outside $\Omega$.
Then $\sum_{l \in Z_{N(r)}} \psi^l_r$ solves $-\operatorname{div } \frac{\eps(u_r)}{2 r} \sum_{l \in Z_{N(r)}}\nabla \psi^l_r = Q_r \rho_r$. \\
Hence, the optimal $\psi_r$ for $E_1(\rho_r,u_r)$ is given by $\psi_r =  \sum_{l \in Z_{N(r)}} \psi^l_r + T_r$ where $T_r$ satisfies $-\operatorname{div } \frac{\eps(u_r)}{2 r} \nabla T_r = 0$
and corrects the boundary conditions appropriately.
From \cite{Littman1963}, we know that there exists a Green's function $K_r$ for $-\operatorname{div}_x \frac{\eps(u_r)}{r} K_r(x,y) = \delta_{y}$ satisfying $K_r(x,y) \leq C r|x-y|^{-1}$, where the constant does
not depend on $r$ (in fact the constant only depends on the maximal values of $\eps(u_r)$). 
Consequently, there exist constants $L>0,C>0$ such that for all $l \in Z_n$ we have $|\psi^l_r(x)| \leq C \frac{r}{|x|}$ for all $|x| \geq L r$.
Therefore, we find that for $x \in \partial \Omega$ it holds for $r$ small enough that $|T_r(x)| \leq 2C \alpha(r) \frac{r}{\operatorname{dist}(x,\partial \Omega)} \to 0$ as $r\to 0$.
By the maximum principle, it follows that $T_r \to 0$ uniformly in $\Omega$. \\
This proves for the electrical energy that 
\begin{align}
 E_r^{el}(\rho_r) &= \int_{\Omega} Q_r\rho_r \, \psi_r - \frac{\varepsilon(u_r)}{2r} |\nabla \psi_r|^2 \, dx \\
 &= \frac12 \int_{\Omega} Q_r \rho_r T_r \,dx + \frac12 \int_{\Omega} Q_r \rho_r \sum_{l \in Z_{N(r)}} \psi_i^l \,dx \\
 &= \frac12 \int_{\Omega} Q_r \rho_r \sum_{l \in Z_{N(r)}} \psi_i^l \,dx + o(\alpha(r)). 
\end{align}
Next, we show that the interaction between the different terms is negligible. Let $L>0$ and fix $y \in B_{Lr}(x + s(r)l)$ for some $l \in Z_{N(r)}$.
Then 
\begin{align}
 |\sum_{l' \neq l} \psi^l_r(y)| &= \sum_{k=1}^{\lceil (\log_2 N(r)^{\frac13}) \rceil} \sum_{ 2^{k-1} \leq |l - l'| < 2^k} |\psi^l_r(y)| \\
 &\lesssim \sum_{k=1}^{\lceil \log_2 N(r)^{\frac13}) \rceil} 2^{3k} \frac{r}{s(r) 2^{k-1}} \\
 &\lesssim \frac{r}{s(r)} N(r)^{\frac23} \approx (r \alpha(r))^{\frac23} \to 0 \text{ as } r\to0. 
\end{align}
As the measures $Q_r \rho_r^l$ are compactly supported on scale $r$, this implies 
\begin{align}
 E_r^{el}(\rho_r) &= \sum_{l \in Z_{N(r)}} \int_{\Omega} \frac12 Q_r \rho_r^l \psi^r_l + o(\alpha(r)) \\
 &=  \sum_{l \in Z_{N(r)}} \int_{\R^3} Q_r \rho_r^l \psi^r_l - \frac{\varepsilon(u_r)}{2r} |\nabla \psi_r^l|^2 \,dx + o(\alpha(r)) \label{eq: upestimate el} 
\end{align}
This is not yet the rescaled electrical part of the energy $E_1(\mu)$ since $u_r$ is only locally the translated and rescaled optimal $u$ for $\mu$.
In the following we show that we can localize the optimal $\psi$ to the region where $u_r$ equals the rescaled and translated optimal $u$ without creating too much energy. \\
First, notice that for fixed $l$ the function $u_r(y) = u \left( \frac{y - s(r)l-x}{r} \right)$ is simply the rescaled and translated optimal $u$ for the single measure $\mu$ in a ball with radius $\gamma_r \approx \frac{r^{\frac13}}{2}$ around $x+s(r)l$.\\
For each $\eta >0$ and $r$ small enough we can find a $k \in \{2,\dots, \lceil \eta \gamma_r \rceil \}$ such that 
\[
 \int_{ B_{k\eta^{-1}r}(x + s(r)l) \setminus B_{(k-1) \eta^{-1}r}(x + s(r)l)} \frac{\eps(u_r)}{2r} |\nabla \psi_l^r |^2 \, dx \leq C \eta r^{\frac13} \int_{\R^3} \frac{\varepsilon(u_r)}{2r} |\nabla \psi_l^r|^2 \,dx.
\]
If we let $\varphi_l = 1$ on $B_{(k-1) \eta^{-1}r}(x + s(r)l)$ and $\varphi_l = 0$ on $B_{k\eta^{-1}r}(x + s(r)l)$ such that $|\nabla \varphi_l| \leq \frac{2\eta}{r}$ 
then  
\begin{align*}
 &\int_{B_{k\eta^{-1}r}(x + s(r)l) \setminus B_{(k-1) \eta^{-1}r}(x + s(r)l)} \frac{\eps(u_r)}{2r} |\nabla (\psi_l^r \varphi_l) |^2 \, dx\\
 \leq &2 \int_{ B_{k\eta^{-1}r}(x + s(r)l) \setminus B_{(k-1) \eta^{-1}r}(x + s(r)l)} \frac{\eps(u_r)}{2r} |\nabla \psi_l^r |^2 + \frac{\eps(u_r)}{2r} |\psi_l^r|^2 \frac{4 \eta^2}{r^2}\, dx
\end{align*}
Using that $|\psi_l^r| \lesssim \frac{r}{|y-x - s(r)l|}$ we derive that
\[
 \int_{ B_{k\eta^{-1}r}(x + s(r)l) \setminus B_{(k-1) \eta^{-1}r}(x + s(r)l)} \frac{\eps(u_r)}{2r} |\nabla (\psi_l^r \varphi_l) |^2 \, dx \leq \eta r^{\frac13} \int_{\R^3} \frac{\varepsilon(u_r)}{2r} |\nabla \psi_r^l|^2 \,dx + C\eta.
\]
As $\psi_l^r \varphi_l$ is now supported in $B_{\gamma_r}(x+s(r)l)$ in which $u_r$ just equals the translated and rescaled optimal $u$ for $\mu$, we find that for $r$ small enough
\begin{align*}
 \int_{\R^3} Q_r \rho_r^l \psi^r_l - \frac{\varepsilon(u_r)}{2r} |\nabla \psi_r^l|^2 \,dx \leq \frac{1}{1 - \eta r^{\frac13}} E_1(\mu,u;B_{\frac{\gamma_r}{r}}(0)) + \eta \leq \frac{1}{1 - \eta r^{\frac13}} E_1(\mu,u;\R^3) + \eta.
\end{align*}
For the last inequality simply notice that a competitor on $B_{\frac{\gamma_r}{r}}(0)$ can always be extended by zero to a competitor on $\R^3$.
Hence, we observe that combining \eqref{eq: upestlocal}, \eqref{eq: upestimate el}, and noticing that $E_1(\mu,u;\R^3) = E_1(\mu;\R^3)$ by definition of $u$ yields
\[
\limsup_{r \to 0} \frac1{\alpha(r)} E_r(\rho_r) \leq \lim_{r\to 0} \frac{N(r)}{\alpha(r)} (E_1(\mu) + \eta + \delta) = \frac1z (E_1(\mu) +\eta+ \delta) \leq \varphi(\xi) + (1+1/ z)(\delta + \eta).
\]
As $\mu$ is close to $\rho$ in total variation, we can find a diagonal sequence satisfying the desired $\limsup$-inequality.

 \underline{\textbf{Step 2:}} $\rho = \sum_{i=1}^L \xi_i \delta_{x_i}$ for some $\xi \in \R^N, x_i\in \Omega$. \\
Using the approximating sequences for each Dirac mass, one can show similarly to step 1 that the interaction between the different sequences is negligible in the limit as $r\to0$.
 
 \underline{\textbf{Step 3:}} The general case $\rho \in \mathcal{M}(\Omega;\R^N)$. \\
The general case follows as weighted sums of Dirac masses are energy-dense in $\mathcal{M}(\Omega;\R^N)$, see also the proof of the upper bound in the case $B \neq 0$.
\end{proof}

\subsection{The Critical Regime}\label{sec: criticalB}

In this section we assume that $\alpha(r) r \to \alpha \in (0,\infty)$.
Here, we need the further technical assumption of well-separateness of solutes to prove the $\Gamma$-convergence result.
We define the admissible solute distributions by
\[
 \tilde{\mathcal{A}}_r(\Omega) = \left\{ \sum_{i=1}^N \sum_{j=1}^{m_i} e_i \delta_{x^{i,j}} \in \mathcal{A}_r: |x^{i,j} - x^{i',j'}| \geq 2\delta_r, \dist(x^{i,j},\partial \Omega) \geq \delta_r \right\},
\]
where $\delta_r \to 0$ such that $\frac{\delta_r}{r} \to \infty$ and $\delta_r^3 \alpha(r) \to 0$.\\
We then define the energy $\tilde{E}_r: \mathcal{M}(\Omega;\R^N) \rightarrow \R \cup \{\infty\}$ by
\[
 \tilde{E}_r(\rho) = \begin{cases}
                E_r(\rho) &\text{if } \rho \in \tilde{\mathcal{A}}_r(\Omega), \\
                +\infty &\text{else}.
               \end{cases}
\]

In the following we show Theorem \ref{thm: Bcrit}, i.e., we prove that the rescaled energy $\frac{\tilde{E}_{r}}{\alpha(r)}$ $\Gamma$-converges with respect to vague convergence of $\frac{\rho_r}{\alpha(r)}$ in $\mathcal{M}(\Omega;\R^N)$ to the energy $E^{crit}: \mathcal{M}(\Omega,\R^N) \rightarrow [0,\infty]$ defined by 
\[
E^{crit}(\rho) = \begin{cases} \frac{\alpha}{2 \eps(1)} \left\| Q_0 \rho \right\|^2_{H^{-1}} + \sum_{i=1}^N E_1(e_i \delta_0;\R^3) |\rho^i|(\Omega) &\text{ if } Q_0 \rho \in H^{-1}(\Omega;\R^M)  \text{ and } \rho \text{ is a } \\&\text{ nonnegative measure in each component}, \\ +\infty &\text{ else.}\end{cases}
\]
The proof will be given in the Propositions \ref{prop: critlower} and \ref{prop: critupper}.\\
Again, we start with the compactness result.
\begin{proposition}[Compactness]\label{prop: compcrit}
Let $\alpha(r) \to \infty$ such that $r \alpha(r) \to \alpha \in (0,\infty)$.
Moreover, let $\{\rho_r\}_{r>0} \subseteq \mathcal{M}(\Omega;R^N)$ such that $\sup_r \frac1{\alpha(r)} E_{r}(\rho_r) \leq C < \infty$.
If $E_1(e_i \delta_0;\R^3)>0$ for all $i=1,\dots,N$ then there exists a (not relabeled) subsequence and $\rho \in \mathcal{M}(\Omega)$ such that
\[
\frac{\rho_r}{\alpha(r)} \weakstar \rho \text{ in } \mathcal{M}(\Omega;\R^N) \text{ and }  \frac{Q_r\rho_r}{\alpha(r)} \rightharpoonup Q_0 \rho \text{ in } H^{-1}(\Omega).
\]
\end{proposition}
\begin{proof}
As $\sup_r \frac1{\alpha(r)} E_{r}(\rho_r) \leq C$, clearly $\rho_r \in \tilde{\mathcal{A}}_{r}(\Omega)$ and hence $\tilde{E}_r(\rho_r) = E_r(\rho_r)$.
We write $\rho_r = \sum_{i=1}^{N} \sum_{j=1}^{m_r^i} e_i \delta_{x_r^{i,j}}$. \\
Fix $0<\eta <\frac13 \min_i E_1(e_i \delta_0;B_L(0))>0$ and, by Lemma \ref{lemma: optimalprofiles}, let $L>0$ such that $E_1(e_i \delta_0;B_L(0)) \geq E_1(e_i \delta_0;\R^3) - \eta$ and $\int_{\mathbb{R}^3\setminus B_L(0)} \1_{\{U_i < 0\}} U_i(x) ,dx \geq -\eta$.
Moreover, let $\psi^i \in H_0^1(B_L(0))$ such that 
\[
E_1(e_i \delta_0;B_L(0)) = \inf_u E_1(e_i \delta_0 u,\psi^i;B_L(0))                                                    
\]
which we extend by $0$ to $\R^3$.
We define $\psi_r^{i,j}(x) = \psi^i\left( \frac{x - x_r^{i,j}}{r} \right)$ and $\psi_r = \sum_{i,j} \psi_r^{i,j}$.
For $u_r:\Omega \rightarrow \{0,1\}$ such that $E_r(\rho_r) = E_r(\rho,u_r)$ we find
\begin{align}
E_{r}(\rho_r) &\geq \int_{\Omega} \geq \sum_{i=1}^{N}\sum_{j=1}^{m_r^i} E_r(e_i \delta_{x_r^{i,j}},u_r,\psi_r^{i,j};B_{Lr}(x_r^{i,j})) - \int_{\R^3\setminus B_L(\supp(\rho_r))} U_{\rho_r,1}^- (x) \, dx \\
&\geq  \sum_{i=1}^{N}|\rho_r^i|(\Omega) \left(E_r(e_i \delta_{x_r^{i,j}};B_{Lr}(x_r^{i,j}))-\eta \right) \\
&= \sum_{i=1}^{N}|\rho_r^i|(\Omega) \left(E_1(e_i \delta_{0};B_{L}(0))-\eta \right) \\
&\geq \sum_{i=1}^{N}|\rho_r^i|(\Omega) \left(E_1(e_i \delta_{0};\R^3)-2\eta \right) \\
&\geq c \sum_{i=1}^{N} |\rho_r^i|(\Omega)
\end{align}
Hence, $\frac{\rho_r}{\alpha(r)}$ is bounded in $\mathcal{M}(\Omega;\R^N)$. \\
On the other hand, for $\psi_r \in H_0^1(\Omega)$ solving $\frac{-\eps(1)}{r} \Delta \psi_r = Q_r \rho_r$ and an optimal $u_r$ for $\rho_r$ we find that
\begin{align*}
 E_r(\rho_r) &\geq -\int_{\R^3} U_{\rho_r,1}^- (x) \,dx + \int_{\Omega} Q_r \rho_r - \frac{\eps(u_r)}{2r} |\nabla \psi_r|^2 \, dx \\
 &\geq - C \alpha(r) + \int_{\Omega} Q_r \rho_r - \frac{\eps(1)}{2r} |\nabla \psi_r|^2 \, dx \\\
 &\geq - C \alpha(r) + \frac{r}{2 \eps(1)} \| Q_r \rho_r\|_{H^{-1}}^2.
\end{align*}
Hence, $\| Q_r \rho_r\|_{H^{-1}}^2 \leq C \frac{\alpha(r)}{r} \leq \tilde{C} \alpha(r)^2$.
Therefore, there exists a (not relabeled) subsequence such that $\frac{\rho_r}{\alpha(r)}$ converges vaguely in $\mathcal{M}(\Omega;\R^N)$ to some $\rho \in \mathcal{M}(\Omega;\R^N)$ ---which is clearly nonnegative---
and $\frac{Q_r \rho_r}{\alpha(r)}$ converges weakly in $H^{-1}(\Omega)$ to some $f \in H^{-1}(\Omega)$ .
Testing shows $f = Q_0 \rho$. 
\end{proof}

\begin{proposition}\label{prop: critlower}
Let $\Omega \subseteq \R^3$ open bounded with Lipschitz boundary and $\alpha(r) \to \infty$ such that $r \alpha(r) \to \alpha \in (0,\infty)$. 
Moreover, let $\rho \in \mathcal(\Omega)$ and $(\rho_r)_r \subseteq \mathcal{M}(\Omega)$. 
Assume that $\frac{\rho_r}{\alpha(r)} \weakstar \rho$ in $\mathcal{M}(\Omega;\R^N)$.
Then
\[
\liminf_{n \to \infty} \frac1{\alpha(r)} E_{r}(\rho_r) \geq \frac{\alpha}{2\varepsilon(1)} \| \rho \|_{H^{-1}(\Omega)}^2 + \sum_{k=1}^N E_1(e_i \delta_0;\R^3) |\rho^i|(\Omega).
\] 
\end{proposition}
\begin{proof}
We may assume that $\liminf_{r \to 0} \frac1{\alpha(r)} E_{r}(\rho_r) = \lim_{r \to 0} \frac1{\alpha(r)} E_{r}(\rho_r)$ and $\sup \frac1{\alpha(r)} E_{r}(\rho_r) < \infty$.
Then by the compactness statement we derive (for a not-relabeled subsequence) the weak convergence $\frac{Q_r \rho_r}{\alpha(r)} \rightharpoonup Q_0 \rho$ in $H^{-1}(\Omega)$.

Moreover, $\rho_r \in \tilde{\mathcal{A}}_{r}(\Omega)$. 
We write $\rho_r = \sum_{i=1}^{N} \sum_{j=1}^{m_r^i} e_i \delta_{x_r^{i,j}}$.
Let $u_r: \Omega \rightarrow \{0,1\}$ be such that $E_{r}(\rho_r) = E_{r}(\rho_r,u_r)$. \\
Fix $\eta >0$.
By Lemma \ref{lemma: optimalprofiles}, there exists $L > 1$ such that for all  $i\in\{1,\dots,N\}$ it  holds 
\[
E_1(e_i \delta_0;B_L(0)) \geq E_1(e_i \delta_0;\R^3) - \eta \text{ and } \int_{\mathbb{R}^3\setminus B_L(0)} \1_{\{U_i < 0\}} U_i(x) ,dx \geq -\eta. 
\]
Next, let $\psi^{i}$ such that $E_1(e_i \delta_0; B_L(0)) = \inf_u E_1(e_i \delta_0,u,\psi^i;B_L(0))$. 
We extend $\psi^i$ by $0$ and define $\psi_r^{i,j} = \psi^i\left(\frac{\cdot - x_r^{i,j}}{r} \right)$.

Moreover, we define $\psi_r = \sum_{i=1}^N \sum_{j=1}^{m^i_r} \psi_r^{i,j} + r \alpha(r) R$ where $R \in H_0^1(\Omega)$ solves $-\eps(1) \Delta R = Q_0 \rho$.
As the supports of the $\psi_r^{i,j}$ are disjoint for $r$ small enough, we observe
\begin{align}
E_{r}(\rho_r) \geq &E_{r}(\rho_r,u_r,\psi_r) \nonumber \\ 
\geq &\sum_{i=1}^N \sum_{j=1}^{m_r^i} \left[E_{r}( e_i \delta_{x^{i,j}_r}, u_r, \psi_r^{i,j}; B_{Lr}(x_r^{i,j})) - r^{-3} \int_{\Omega \setminus B_{Lr}(x^{i,j}_r)} U_{LJ}^i \left(\frac{x-x_r^{i,j}}{r} \right) \, dx \right] \nonumber \\
+ &\int_{\Omega} Q_r \rho_r \, r \alpha(r) R - \frac{\eps(u_r)}{2r} r^2 \alpha(r)^2 |\nabla R|^2 \,dx - \int_{\Omega} \frac{\eps(u_r)}{r} \nabla \left(\sum_{i=1}^N \sum_{j=1}^{m^i_r} \psi_r^{i,j} \right) \cdot \nabla (r \alpha(r) R) \, dx \nonumber \\
\geq &\sum_{i=1}^N \left[(E_{1}( e_i \delta_0; B_{L}(0)) -\eta) \, |\rho_r^i|(\Omega) \right] + \int_{\Omega} Q_r \rho_r \, r \alpha(r) R - \frac{\eps(1)}{2r} r^2 \alpha(r)^2 |\nabla R|^2 \,dx \\
&- \int_{\Omega} \frac{\eps(u_r)}{r} \nabla \left(\sum_{i=1}^N \sum_{j=1}^{m^i_r} \psi_r^{i,j} \right) \cdot \nabla (r \alpha(r) R) \, dx \label{eq: mainestimatelower}
\end{align}
First, note that by the choice of $L$ we have 
\begin{align}
&\sum_{i=1}^N (E_{1}( e_i; B_{L}(0))-\eta) |\rho_r^i|(\Omega) \geq \sum_{i=1}^N (E_{1}( e_i \delta_0; \R^3)-2\eta) |\rho_r^i|(\Omega). \label{eq: estlower1}
\end{align}
For the second term in \eqref{eq: mainestimatelower}, since $r \alpha(r) \to \alpha$, we find as $r \to 0$
\begin{equation}
\frac1{\alpha(r)} \int_{\Omega} Q_r \rho_r \, r \alpha(r) R - \frac{\eps(1)}{2} \alpha(r)^2 r |\nabla R|^2 \, dx \longrightarrow  \frac{\alpha}2 < Q_0 \rho,R>_{H^{-1},H_0^1} = \frac{\alpha}{2\eps(1)} \| Q_0 \rho \|_{H^{-1}}^2. \label{eq: estlower2}
\end{equation}
For the third term in \eqref{eq: mainestimatelower}, we notice that similarly to the compactness proof, using the test function $\tilde{\psi}_r = \sum_{i=1}^N\sum_{j=1}^{M^i_n} \psi_r^{i,j}$ shows for
$\eta \leq \frac12 \min_i E_1(e_i \delta_0; B_L(0))$ that
\[
C \alpha(r) \geq E_{r}(\rho_r) \geq \frac12 \sum_{i=1}^M E_{1}( e_i \delta_0; B_{L}(0)) |\rho_r^i|(\Omega) \geq c  \sum_{i=1}^N \sum_{j=1}^{m_r^i} \int_{B_{r}(x_r^{i,j})} \frac{\eps(0)}{r} |\nabla \psi_r^{i,j}|^2 \, dx.
\]
As $r\alpha(r) \to \alpha$, it follows that $\sum_{i=1}^{N}\sum_{j=1}^{m^i_r} \psi_r^{i,j}$ is a bounded sequence in $H_0^1(\Omega)$. 
As the measure of the support of the function $\sum_{i=1}^{M}\sum_{j=1}^{m^i_r} \psi_r^{i,j}$ goes to zero as $r\to 0$, it follows that $\sum_{i=1}^{M}\sum_{j=1}^{m^i_r} \psi_r^{i,j} \rightharpoonup 0$ in $H_0^1(\Omega)$.
Then, as $u_r \to 1$ boundedly in measure, it follows for the last term in \eqref{eq: mainestimatelower} that 
\begin{equation} \label{eq: estlower3}
\frac1{\alpha(r)} \int_{\Omega} \frac{\eps(u_r)}{2r} \left(\nabla \sum_{i=1}^{N}\sum_{j=1}^{m^i_r} \psi_r^{i,j}\right) \cdot r\alpha(r) \nabla R \,dx \rightarrow 0 \text{ as } r\to 0.
\end{equation}
Combining \eqref{eq: mainestimatelower}, \eqref{eq: estlower1}, \eqref{eq: estlower2}, and \eqref{eq: estlower3} shows the claimed lower bound after sending $\eta \to 0$.
\end{proof}

In order to prove the existence of a recovery sequence, we first prove the following simple approximation result which we will also use in the supercritical regime.

\begin{lemma} \label{lemma: approxrho}
Let $\Omega \subseteq \mathbb{R}^3$ and $\alpha(r) \to \infty$ such that $\alpha(r) r^3 \rightarrow 0$. 
Let $\nu = \xi \1_{E}$ for an open sets $E \Subset \Omega$ and $\xi^i \geq0$ for all $i=1,\dots,M$. 
Moreover, let $\lambda_1,\dots,\lambda_K >0$ and $\xi_1,\dots,\xi_K \in \R^M$ such that $\xi= \sum_{k=1}^K \lambda_k \xi_k$.
Then there exists a sequence of measures $\rho_r = \sum_{k=1}^{K} \sum_{j=1}^{M_{r}^k} \nu_r^{j,k}$ where each $\nu_r^k$ is of the form $\xi_k \delta_{x_n^j}$ such that $\operatorname{dist}(\supp(\nu_r^{j,k}),\supp(\nu_r^{k',j'})) \geq c \alpha(r)^{-\frac13}$ for all $(k,j) \neq (k',j')$ and
\begin{equation} \label{eq: approximationrho}
\frac{\rho_r}{\alpha(r)} \weakstar \nu \text{ in } \mathcal{M}(\Omega) \text{ and } \frac{|M_r^k|}{\alpha(r)} \rightarrow \lambda_k |E|.
\end{equation}

\end{lemma}
\begin{proof}
First, we assume that $K=1$, $\lambda_1 = 1$, and $\xi_1 = \xi$.
Cover $E$ with cubes $Q_x$ with sidelength $\alpha(r)^{-\frac13}$ and centers $x$ in the lattice $\mathcal{L}_r = \alpha(r)^{- \frac13}\Z^3$.
Then define 
\[
\rho_r = \sum_{x \in \mathcal{L}_r s.t. Q_x \subseteq E} \xi_1 \delta_{x}. 
\]
We observe immediately the convergences stated in \eqref{eq: approximationrho}.

For the general case, first approximate $\rho$ weakly in $L^2$ by alternating functions which are of constantly $\xi_k$ on subsets of $E$ of volume fraction $\frac{\lambda_k}{\sum_{k=1}^K \lambda_k}$. The general case then follows by a diagonal argument.
\end{proof}

Now, we are able to prove the upper bound.

\begin{proposition}\label{prop: critupper}
Let $\rho \in \mathcal{M}(\Omega;\R^N)$ such that $ Q_0\rho \in H^{-1}(\Omega)$. 
Then there exists $\{\rho_r\}_{r>0} \subseteq \mathcal{M}(\Omega)$ such that 
$\frac{\rho_r}{\alpha(r)} \weakstar \rho$ in $\mathcal{M}(\Omega;\R^N)$, $\frac{Q_0\rho_r}{\alpha(r)} \rightharpoonup \rho$ in $H^{-1}(\Omega)$, and 
\[
\limsup_{n\to\infty} \frac1{\alpha(r)} E_{r} (\rho_r) \leq \frac{\alpha}{2\varepsilon(1)} \| \rho \|_{H^{-1}(\Omega)}^2 + \sum_{i=1}^N |\rho^i|(\Omega) E_1(e_i \delta_0;\R^3).
\]
\end{proposition}
\begin{proof}
\textbf{Step 1:}\textit{$\rho = \sum_{i=1}^L \xi_i \1_{E_{i}}$ where $\xi_i \cdot e_k \geq 0$ for all $k=1,\dots,N$, and $E_{i} \Subset \Omega$ are open.} \\
We apply Lemma \ref{lemma: approxrho} to each $E_i$, $e_1,\dots,e_N$, and $\lambda_1^i = \xi_i \cdot e_1,\dots,\lambda_N^i = \xi_i \cdot e_N$ to obtain
a sequences of measures $\rho_r = \sum_{i=1}^{L} \sum_{k=1}^N \sum_{j=1}^{M_r^{k,i}} e_k \delta_{x_j^{k,i}}$ such that  
\[
\frac{\rho_r}{\alpha(r)} \weakstar \rho \text{ in } \mathcal{M}(\Omega;\R^N), \text{ and }\frac{|M_r^{j,k}|(\Omega)}{\alpha(r)} \rightarrow \xi^k |E_j|.
\]
Then, by construction we have that for $x \neq y \in \bigcup_{j} \supp(\rho^j_r)$ it holds $|x-y| \geq c \alpha(r)^{-\frac13}$. 
Hence, $\rho_r \in \tilde{A}_r(\Omega)$ for $r$ small enough. \\
Next, let $u^{k}$, $\psi^{k}$ be a minimax pair for $E_1(e_k \delta_0;\R^3)$, in particular $-\operatorname{div }( \eps(u^k) \nabla \psi^k) = Q_1 (e_k\delta_0)$. \\

Moreover, by the usual argument, we can find $R>0$ such that for all $k$ it holds that $u^{k} = 1$ outside $B_R(0)$.
This implies that $|\psi^{k}| \leq \frac{C}{|x|}$ for $|x| \geq 2R$ and also $|\nabla \psi^{k} (x)| \leq \frac{C}{|x|^2}$ for $|x| \geq 2R$.
We define the functions
\[
 u_r(x) = \prod_{j,k,i} u^{k}\left(\frac{x - x_j^{k,i}}{r}\right) \text{ and } \psi_j^{k,i}(x) = \psi^{k}\left(\frac{x - x_j^{k,i}}{r} \right).
\] 
Next, let $\psi^{opt}_r \in H_0^1(\Omega)$ be optimal for $E_{r}(\rho_r,u_r)$ i.e., $-\operatorname{div }( \frac{\eps(u_r)}{r} \nabla \psi^{opt}_r) = Q_{r} \rho_r$.
Moreover, let $\gamma_r \approx \alpha(r)^{-\frac13}$ such that for each two points in $\supp(\nu_r)$ we have $|x-y| \geq \frac{\gamma_r}{2}$. 
We define 
\[
R_r = \nabla \psi^{opt}_r - \sum_{j,k,i} \nabla (\psi_j^{k,i}) \varphi_j^{k,i}, 
\]
where $\varphi_j^{k,i} \in C^{\infty}_c(B_{\gamma_r}(x_j^{k,i})) = 1$ such that $\varphi_j^{k,i} =1 $ on $B_{\gamma_r/2}(x_j^{k,i})$ and $|\nabla \varphi_j^{k,i}| \leq C \gamma_r^{-1}$. \\
We claim that
\begin{enumerate}[label=(\roman*)]
 \item $\frac{\sum_{j,k,i} (\nabla \psi_j^{k,i}) \varphi_j^{k,i}}{\alpha(r) r} \rightharpoonup 0$ in $L^2(\Omega)$, \label{item: upper1}
 \item $\frac{\tilde{\mu}_r}{\alpha(r)} \rightarrow Q_0 \rho$ in $H^{-1}(\Omega)$, where $\tilde{\mu}_r = -\sum_{j,k,i} \frac{\eps(1)}{r} \nabla (\psi_j^{k,i}) \cdot \nabla \varphi_j^{k,i}$, \label{item: upper2}
 \item $\frac{R_r}{r \alpha(r)} \rightarrow \nabla \psi$ strongly in $L^2(\Omega)$. \label{item: upper3}
\end{enumerate}

From the optimality of the $\psi^{k}$ it follows since the supports of the occurring functions are disjoint that $\frac{\sum_{j,k,i} (\nabla \psi_j^{k,i}) \varphi_j^{k,i}}{\alpha(r) r}$ is bounded in $L^2(\Omega)$. 
Moreover, using the bound $|\nabla \psi_j^{k,i}| \leq \frac{Cr}{|x - x_j^{k,i}|^2}$ for all $|x - x_j^{k,i}| \gg r$ one can show that for each $\alpha(r)^{-\frac13} \gg \eta_r \gg r$ the $L^2$-norm of 
$\frac{\sum_{j,k,i} \nabla (\psi_j^{k,i}) \varphi_j^{k,i}}{\alpha(r) r} \1_{\Omega \setminus \bigcup_{j,i,k} B_{\eta_r}(x_j^{k,i})}$ goes to zero.
This shows  \ref{item: upper1}.\\
For \ref{item: upper2}, observe that $-\operatorname{div} (\frac{\eps(u_r)}{r} R_r ) = -\sum_{j,k,i} \frac{\eps(1)}{r} \nabla (\psi_j^{k,i}) \cdot \nabla \varphi_j^{k,i} =: \tilde{\mu}_r$.
Again, one can show using the bound on $\nabla \psi_j^{k,i}$ that $\frac{\tilde{\mu}_r}{\alpha(r)}$ is bounded in $L^2$. 
Simple testing shows that $\frac{\tilde{\mu}_r}{\alpha(r)} \rightharpoonup Q_0\rho$ in $L^2(\Omega)$.
Hence, \ref{item: upper2}.\\
Combining \ref{item: upper1} and \ref{item: upper2}, we find that also $Q_r\rho_r = -\operatorname{div } (\frac{\eps(u_r)}{r \alpha(r)} \nabla \psi^{opt}_r) \rightharpoonup Q_0\rho$ in $H^{-1}$ and consequently $\frac{\psi_r^{opt}}{r \alpha(r)}$ is bounded in $H_0^1$.
Hence, up to a subsequence $\frac{\psi_r^{opt}}{r \alpha(r)} \rightharpoonup \psi \in H_0^1(\Omega)$ which ---using that $\eps(u_r) \rightarrow \eps(1)$ boundedly in measure--- can be shown to be the unique solution to $-\eps(1) \Delta \psi = Q_0 \rho$.
Consequently, it follows from the definition of $R_r$ and \ref{item: upper1} that $\frac{R_r}{r \alpha(r)} \rightharpoonup \nabla \psi$ in $L^2$.\\
Next, we prove that actually $\frac{R_r}{r \alpha(r)} \rightarrow \nabla \psi$ strongly in $L^2$.
Observe
\begin{align}
 \int_{\Omega} \eps(u_r) \left|\frac{R_r}{r \alpha(r)}\right|^2 \,dx &= \int_{\Omega} \eps(u_r) \frac{R_r}{r \alpha(r)} \cdot \frac{1}{r \alpha(r)}( \nabla \psi_r^{opt} - \sum_{j,i,k} (\nabla \psi_j^{k,i}) \varphi_j^{k,i} ) \,dx \\
 &= \int_{\Omega} \frac{\tilde{\mu}_r}{\alpha(r)} \frac{\psi_r^{opt}}{r \alpha(r)} \,dx + \int_{\Omega} \frac{\tilde{\mu}_r}{\alpha(r)} \frac{\sum_{j,i,k} \psi_j^{k,i} \varphi_j^{k,i}}{r \alpha(r)} \,dx
+ \int_{\Omega} \eps(u_r) \frac{R_r}{r \alpha(r)} \cdot \frac{\sum_{j,i,k} \psi_j^{k,i} \nabla \varphi_j^{k,i}}{r \alpha(r)} \,dx. \label{eq: estRn}
 \end{align}
Now, the first term in \eqref{eq: estRn} converges to $<Q_0\rho,\psi>_{H^{-1},H^1_0} = \int_{\Omega} \eps(1) |\nabla \psi|^2$.
Using the bounds on $\psi_j^{k,i}$ a simple computation shows that $\left\| \frac{\sum_{j,i,k} \psi_j^{k,i} \nabla \varphi_j^{k,i}}{r \alpha(r)} \right\|_{L^2}^2 \leq C \alpha(r)^{-\frac23}$.
Hence, H\"older's inequality proves that the last term converges to $0$.
For the second term in \eqref{eq: estRn}, we notice that by what we have already proved above we find that $ \frac{\nabla \left(\sum_{j,i,k} \psi_j^{k,i} \varphi_j^{k,i}\right)}{r \alpha(r)} \rightharpoonup 0$ in $L^2$.
Consequently, also $\frac{\sum_{j,i,k} \psi_j^{k,i} \varphi_j^{k,i}}{r \alpha(r)} \rightharpoonup 0$ in $H_0^1$.
Hence, in the second term we may pass to the weak-strong limit which is $0$.
This shows that 
\[
\int_{\Omega} \eps(u_r) \left|\frac{R_r}{r \alpha(r)}\right|^2 \,dx \rightarrow \int_{\Omega} \eps(1) |\nabla \psi|^2. 
\]
Together with the weak convergence of $\frac{R_r}{r \alpha(r)} \rightharpoonup \nabla \psi$ in $L^2$ this implies
\begin{align*}
 \int_{\Omega} \left| \frac{R_r}{r \alpha(r)} - \nabla \psi \right|^2 \,dx &\leq \int_{\Omega} \eps(u_r) \left| \frac{R_r}{r \alpha(r)} - \nabla \psi \right|^2 \,dx  \\
 &= \int_{\Omega} \eps(u_r) \left| \frac{R_r}{r \alpha(r)} \right|^2 - 2 \eps(u_r) \frac{R_r}{r \alpha(r)} \cdot \nabla \psi + \eps(u_r) |\nabla \psi|^2 \,dx \stackrel{r\to0}{\rightarrow} 0,
\end{align*}
which is \ref{item: upper3}.\\
For the energy we find the estimate
\begin{align}
E_{r}(\rho_r) \leq &E_{r}(\rho_r,u_r,\psi_r^{opt}) \nonumber \\
\leq &\sum_{i=1}^{L} \sum_{k=1}^N M_r^{k,i} \left(E_1(e_k \delta_0,u^{k},\psi^{k};\R^3) -  \int_{\R^3 \setminus B_{\gamma_r / r}(0)} U_{1,e_k \delta_0} \1_{\{U_{1,e_k\delta_0} \leq 0\}} \,dx \right) \\
&+ \int_{\Omega} \frac{\eps(u_r)}{2r} |R_r|^2 \,dx + 2 \int_{\Omega} \frac{\eps(u_r)}{r} R_r \cdot \sum_{j,k,i} (\nabla \psi_j^{k,i}) \varphi_j^{k,i} \,dx. \label{eq: estimateenergycrit}
\end{align}
Here, we used that as $\psi_r^{opt}$ is optimal for $\rho_r$ and $u_r$, $u_r$ is locally the rescaled optimal $u^{k}$, and we can write the electric energies by integration by parts as integrals involving only the gradients.
Since $\frac{M_r^{k,i}}{\alpha(r)} \rightarrow \xi_k^i |E_i|$, it follows using \ref{item: upper1} - \ref{item: upper3} and $\frac{Q_r\rho_r}{\alpha(r)} \rightharpoonup Q_0 \rho$ in $H^{-1}(\Omega)$ that
\begin{align}
\limsup_{r \to 0} \frac1{\alpha(r)} E_{r}(\rho_r) &\leq \sum_{i=1}^L \sum_{k=1}^{N} \xi_k^i |E_i| \, E_1(e_k \delta_0;\R^3)+ \alpha <Q_0 \rho, R>_{H^{-1},H_0^1} - \alpha \int_{\Omega}  \frac{\eps(1)}{2} |\nabla \psi|^2 \, dx \\
&= \sum_{k=1}^N |\rho^k|(\Omega) E_1(e_k \delta_0;\R^3) + \frac{\alpha}{2 \eps(1)} \| Q_0\rho \|_{H^{-1}}^2. \label{eq: estimateenergycrit2}
\end{align}

\noindent \textbf{Step 2: }\textit{Energy density} \\
Let $\rho$ be a measure in $\mathcal{M}(\Omega;\R^N)$ such that $\rho^i \geq 0$ for all $i\in\{1,\dots,N\}$ and $Q_0\rho \in H^{-1}(\Omega)$. 
Using the continuity properties of $E^{crit}$ and considering $\rho + \eta (1,\dots,1) \mathcal{L}^3_{|\Omega}$ for $\eta>0$, we may assume without loss of generality that $\rho^i > 0$ for all $i \in \{1,\dots,N\}$.\\
As $\mathcal{M}(\Omega) \hookrightarrow W^{-1,p}(\Omega)$ for all $1\leq p < \frac32$, there exists for each $i \in\{1,\dots,N\}$ $\psi^i \in W_0^{1,\frac43}(\Omega)$ such that $-\Delta \psi_i = \rho^i$. 
On the other hand, the equation $-\Delta \psi = Q_0 \rho$ has a solution $\psi \in H_0^1(\Omega)$.
By uniqueness, $\psi  = \sum_{i=1}^N \left(\int_{\R^3} \phi_i\,dx\right) \psi^i$.
Using local reflection over the boundary, we can extend $\psi^i$ to a function $\tilde{\psi}^i$ in a neighborhood $\tilde{\Omega}$ of $\Omega$ such that in $\tilde{\Omega}$ it still holds that $-\left(\int_{\R^3} \phi_i\,dx\right)\Delta \tilde{\psi}^i$ is a nonnegative measure such that $|\Delta \psi^i|(\partial \Omega) = 0$ and $\tilde{\psi} = \sum_{i=1}^N \left(\int_{\R^3} \phi_i\,dx\right)\tilde{\psi}^i \in H^1(\tilde{\Omega})$.\\ 
Next, let $\eta_n$ be a standard mollifier and define the measures $\rho^i_n = \Delta \tilde{\psi}^i * \eta_n \in \mathcal{M}(\Omega)$. 
Then $\rho_n \weakstar \rho$ in $\mathcal{M}(\Omega;R^N)$ such that $|\rho^i_n|(\Omega) \rightarrow |\rho|(\Omega)$.
Moreover, $\tilde{\psi} * \eta_n \rightarrow \psi$ strongly in $H^{1}(\Omega)$. 
In particular, $Q_0\rho_n \rightarrow Q_0\rho$ strongly in $H^{-1}(\Omega)$. 
By Reshetnyaks' theorem, we know that the energy $E^{crit}$ is continuous with respect to the convergence established above.\\  
Hence, the measures $\rho \in \mathcal{M}(\Omega;\R^N)$ such that $Q_0\rho \in H^{-1}(\Omega)$ and $\rho^i \geq 0$ with a smooth density with respect to the Lebesgue measure are energy-dense in $\mathcal{M}(\Omega;\R^N)$. 
On the other hand, those measures can be approximated strongly in $L^2$ by measures as in step 1. 
This shows that the measures from step 1 are energy dense in $\mathcal{M}(\Omega;\R^N)$.

\end{proof}
\begin{remark}
Combining the argument in the upper bound in the subcritical regime and the proof above, one can see that it is also possible to prove the upper bound for the more general energy $E_r$ i.e., the existence of a recovery sequence $\rho_r$ such that
\[
\limsup_{r\to0} \frac{E_r(\rho_r)}{\alpha(r)} \leq \int_{\Omega} \varphi\left(\frac{d\rho}{d|\rho|}\right) \, d|\rho| + \frac{\alpha}{2\eps(1)} \|Q_0 \rho\|_{H^{-1}}^2,
\]
where $\varphi$ is the self-energy defined in \eqref{eq: defselfenergy}.
\end{remark}

\subsection{The Supercritical Regime}\label{sec: supercriticalB}

Finally, in this section we assume that $\alpha(r) r \to \infty$ and $\alpha(r) r^3 \to 0$.
By the heuristics discussed in Subsection \ref{sec: heuristics} the electric interaction dominates the energy.
We prove Theorem \ref{thm: Bsuper} in Propositions \ref{prop: supercomp}, \ref{prop: superlower}, and \ref{prop: superupper}.

\begin{proposition}[Compactness]\label{prop: supercomp}
Let $r \to 0$ and $\alpha(r) \to \infty$ such that $r \alpha(r) \to \infty$. 
Let $(\rho_r)_r$ be a sequence in $\mathcal{M}(\Omega;\R^N)$ such that $E_{r}(\rho_r) \leq C \alpha(r)^2 r$. 
If $a>0$ is large enough then there exists a subsequence and $\mu \in H^{-1}(\Omega)$ such that 
\[
\frac{Q_{r} \rho_r }{\alpha(r)} \rightharpoonup \mu \text{ in } H^{-1}(\Omega).
\]
\end{proposition}
\begin{proof}
By our assumptions, it follows that $\rho_r \in \mathcal{A}_{r}(\Omega)$ for all $r$.
Let $u_r:\Omega \to \{0,1\}$ be optimal for $\rho_r$ and $\psi_r\in H_0^1(\Omega)$ the solution to $-\frac{\eps(1)}{r} \Delta \psi_r  = Q_{r}\rho_r$.
Hence, $E_{r}(\rho_r) \geq E_{r}(\rho_r,u_r,\psi_r)$.
For $a>0$ large enough, we can drop all but the electric term and obtain
\begin{align}
 E(\rho_r) &\geq \int_{\Omega} Q_{r}\rho_r \psi_r - \frac{\eps(u_r)}{2 r} |\nabla \psi_r|^2 \, dx \\&
 \geq \int_{\Omega} Q_{r}\rho_r \psi_r - \frac{\eps(1)}{2 r} |\nabla \psi_r|^2 \, dx \\ 
 &= -\frac{r}{2\eps(1)} <Q_{r} \rho_r, \Delta^{-1}( Q_{r} \rho_r)>_{H^{-1},H_0^1} =  \frac{r}{\eps(1)} \| Q_{r} \rho_r \|_{H^{-1}}^2.
\end{align}
For the inequality, we simply used that $\eps(1) \geq \eps(0)$.
As $E(\rho_r) \leq C r \alpha(r)^2$, we find that $\frac{Q_{r} \rho_r}{\alpha(r)}$ is bounded in $H^{-1}$ and hence converges up to a subsequence to some $\mu \in H^{-1}(\Omega)$.

\end{proof}

\begin{proposition}\label{prop: superlower}
Let $r \to 0$ and $\alpha(r) \to \infty$ such that $r \alpha(r) \to \infty$. 
Let $(\rho_r)_r$ be a sequence in $\mathcal{M}(\Omega)$ and $\mu \in H^{-1}(\Omega)$ such that $\frac{Q_{r}\rho_r }{\alpha(r)} \rightharpoonup \mu \text{ in } H^{-1}(\Omega)$. 
Then
\[
\liminf_{n \to \infty} \frac{1}{\alpha(r)^2 r} E_{r}(\rho_r) \geq \frac{1}{2 \eps(1)} \| \mu \|^2_{H^{-1}(\Omega)}
\]
\end{proposition}
\begin{proof}
We may assume that $\sup \frac{1}{\alpha(r)^2 r} E_{r}(\rho_r) < \infty$. 
Arguing exactly as in the compactness result above we find that
\[
 \frac1{\alpha(r)^2 r} E(\rho_r) \geq \frac{1}{2\eps(1)} \left\| \frac{Q_{r} \rho_r}{\alpha(r)} \right\|_{H^{-1}}^2.
\]
As $\frac{Q_{r}\rho_r }{\alpha(r)} \rightharpoonup \mu \text{ in } H^{-1}(\Omega)$ the lower bound follows directly from the semi-continuity of the $H^{-1}$-norm.

\end{proof}

\begin{proposition}\label{prop: superupper}
Let $\mu \in H^{-1}(\Omega)$. 
Then there exists a sequence $\rho_r \in \mathcal{A}_{r}(\Omega)$ such that $\frac{Q_{r}\rho_r}{\alpha(r)} \rightharpoonup Q_0\rho$ in $H^{-1}(\Omega)$ and
\[
 \limsup_{n \to \infty} \frac1{r \alpha(r)^2} E_{r}(\rho_r) \leq \frac1{2 \varepsilon(1)} \| \mu \|_{H^{-1}(\Omega)}^2. 
\]
\end{proposition}

\begin{proof}
 By standard density arguments we may assume that $\mu = \sum_{j=1}^L c_j \1_{E_{j}}$, where the $c_j \in \mathbb{R}$ and the $E_{j} \Subset \Omega$ are open. 
 Let $\rho = \sum_{j=1}^L \xi_j \1_{E_j} \in \mathcal{M}(\Omega;\R^N)$ such that $Q_0 \rho = \mu$. \\
Now, we can argue as in the upper bound in the critical regime step 1 to find a measure $\rho_r \in \mathcal{A}_{r}$ such that $\frac{Q_{r} \rho_r}{\alpha(r)} \rightharpoonup Q_0 \rho$
and obtain the corresponding version of the estimate \eqref{eq: estimateenergycrit}
\begin{align}
E_{r}(\rho_r) \leq &\sum_{j=1}^{L} \sum_{k=1}^N M_r^{j,k} \left(E_1(e_k \delta_0;\R^3) -  \int_{\R^3 \setminus B_{\gamma_r / r}(0)} U_{1,\delta_0} \1_{\{U_{1,\delta_0} \leq 0\}} \,dx \right) \\
&+ \int_{\Omega} \frac{\eps(u_r)}{2r} |R_r|^2 \,dx + 2 \int_{\Omega} \frac{\eps(u_r)}{r} R_r \cdot \sum_{j,k,i} (\nabla \psi_j^{k,i}) \varphi_j^{k,i} \,dx.
\end{align}
where $\frac{R_r}{\alpha(r) r)} \rightarrow \nabla \psi$ in $L^2(\Omega)$ where $\psi \in H_0^1(\Omega)$ solves $-\eps(1) \Delta \psi = Q_0 \rho$, $u_r \rightarrow 1$ in measure, $\frac{M_r^{j,k}}{\alpha(r)^2 r} \to 0$, and $\frac{\sum_{j,k,i} (\nabla \psi_j^{k,i}) \varphi_j^{k,i}}{\alpha(r)r} \rightharpoonup 0$ in $L^2(\Omega)$.
Hence, as $r \alpha(r) \to \infty$, we find that
\begin{align}
\limsup_{r\to0} \frac1{\alpha(r)^2 r} E_{r}(\rho_r) &\leq \int_{\Omega}  \frac{\eps(1)}{2} |\nabla \psi|^2 \, dx =\frac{1}{2 \eps(1)} \| Q_0\rho \|_{H^{-1}}^2.
 \end{align}

\end{proof}

\section{Conclusion and Future Directions}\label{sec: conclusion}
We studied the asymptotic behavior of a sharp-interface model for the solvation of molecules in an implicit solvent as the number of solute molecules and the size of the surrounding box go to infinity.
For the model including $B$ as in \eqref{eq: defB} we proved a screening effect, i.e., the limit energy is purely local. 
In the case $B=0$, we identified the competing local and nonlocal interaction terms and the corresponding regimes. \\
Work in progress is to get rid of the extra assumption of well-separateness.
As in the subcritical regime, one has to find good clusters of molecules whose interaction with the neighboring clusters is negligible in the limit. 
Then one could split the electrical energy into a local self-energy and a far-field interaction.
The approach of considering the quantity $\sum_{i\neq j} \frac{1}{|x_i-x_j|}$ we used in the subcritical regime cannot be modified easily.
The problem is that by considering the absolute value of the interactions we cannot distinguish between the energetically allowed configuration of dipoles concentrating on a line and the energetically very expensive configuration of molecules of the same charge concentrating on a line.  \\
Moreover, we would like to study the $\Gamma$-convergence result with respect to convergence of measures in the dual of $C_b^0(\Omega)$ opposed to $C_0^0(\Omega)$ to capture also boundary effects.
The $\Gamma$-limit of $E_r / \alpha(r)$, say in the case of $B\neq 0$ and $\partial \Omega \in C^1$, is expected to be of the form
\[
\int_{\Omega} \varphi\left( \frac{d\rho}{d|\rho|} \right) \, d|\rho| + \int_{\partial \Omega} \varphi_{\partial}\left(\nu, \frac{d\rho}{d|\rho|}\right) \, d|\rho|,
\] 
where $\varphi_{\partial}: S^2 \times \R^N \rightarrow \R$ is the self-energy density on the boundary which is essentially given as the subadditive, $1$-homogeneous envelope of the energy for a single cluster of molecules in a half-space.
If we allow further to rotate all molecules, we can also eliminate the dependence on $\nu$.

\section*{Acknowledgments}

The authors wish to acknowledge the Center for Nonlinear Analysis where this work was carried out.The research of Janusz Ginster was funded by the National Science Foundation under NSF PIRE Grant No. OISE-0967140.

The authors are deeply grateful to Irene Fonseca and Giovanni Leoni for bringing the topic to their attention and for many fruitful discussions.

\bibliographystyle{abbrv}
\bibliography{solvation}

\end{document}